\documentclass[11pt]{article}
\usepackage[utf8]{inputenc}
\usepackage{bbm}
\usepackage{amsfonts,amsmath,amsthm,amsbsy,amssymb,parskip}
\usepackage[margin=1in]{geometry}
\usepackage{mathtools}  
\usepackage{wasysym}
\usepackage{bm}
\usepackage{titling}
\usepackage{csquotes}
\usepackage{comment}
\usepackage{ytableau}
\usepackage{ stmaryrd }
\usepackage{tikz-cd}
\usepackage{hyperref}
\newtheorem{theorem}{Theorem}[section]
\newtheorem{lemma}[theorem]{Lemma}

\newtheorem{corollary}[theorem]{Corollary}
\theoremstyle{definition}
\newtheorem{definition}[theorem]{Definition}
\theoremstyle{remark}
\newtheorem{remark}[theorem]{Remark}
\theoremstyle{definition}
\newtheorem{exmp}{Example}[section]
\usetikzlibrary{calc}
\usetikzlibrary{nfold}
\tikzset{curve/.style={settings={#1},to path={(\tikztostart)
    .. controls ($(\tikztostart)!\pv{pos}!(\tikztotarget)!\pv{height}!270:(\tikztotarget)$)
    and ($(\tikztostart)!1-\pv{pos}!(\tikztotarget)!\pv{height}!270:(\tikztotarget)$)
    .. (\tikztotarget)\tikztonodes}},
    settings/.code={\tikzset{quiver/.cd,#1}
        \def\pv##1{\pgfkeysvalueof{/tikz/quiver/##1}}},
    quiver/.cd,pos/.initial=0.35,height/.initial=0}
\hypersetup{
    colorlinks=true,
    linkcolor=black,
    filecolor=magenta,      
    urlcolor=cyan,
    pdftitle={Overleaf Example},
    pdfpagemode=FullScreen,
    citecolor=blue
    }

\usepackage{lipsum}
\usepackage{fancyhdr}
\tikzcdset{scale cd/.style={every label/.append style={scale=#1},
    cells={nodes={scale=#1}}}}

\usepackage{xcolor}

\definecolor{myred}{RGB}{214,92,92}
\definecolor{myblue}{RGB}{59,62,247}
\definecolor{mygreen}{RGB}{92,214,92}

\newcommand{\tmpx}{}

\fancypagestyle{sec}{\lhead{\tmpx}}

\newcommand{\BEA}{\begin{eqnarray*}}
\newcommand{\EEA}{\end{eqnarray*}}
\newcommand{\BE}{\begin{enumerate}}
\newcommand{\EE}{\end{enumerate}}

\newcommand{\bmat}{\begin{pmatrix}}
\newcommand{\emat}{\end{pmatrix}}
\newcommand{\bp}{\begin{proof}}
\newcommand{\epp}{\end{proof}}

\newcommand{\B}{\mathcal{C}}

\newcommand{\Z}{\mathbb{Z}}

\newcommand{\F}{\mathbbm{F}}

\newcommand{\kt}{\mathbbm{k}^{\times}}
\newcommand{\C}{\mathbbm{C}}
\newcommand{\chit}{\Tilde{\chi}}

\newcommand{\psit}{\Tilde{\psi}}
\newcommand{\psitq}{\Tilde{\psi}^q}

\newcommand{\eqbr}{\text{EqBr}}
\newcommand{\brpic}{\text{BrPic}}

\newcommand{\inv}{^{-1}}

\newcommand{\Hom}{\text{Hom}}

\newcommand{\obj}{\text{Obj}}

\newcommand{\vek}{\text{Vec}}
\newcommand{\Aut}{\text{2-Aut}_\otimes}

\newcommand{\aut}{\text{Aut}_\otimes}
\newcommand{\id}{\text{id}}
\newcommand{\Gr}{\text{-Gr}}
\newcommand{\BK}{\mathcal{C}\rtimes_\omega K}
\newcommand{\q}{\widehat{q}}
\newcommand{\qr}{\widehat{qr}}
\newcommand{\ar}{\widehat{r}}
\newcommand{\s}{\widehat{s}}
\newcommand{\tee}{\widehat{t}}

\tikzset{
  trim node/.default=1cm,
  trim node/.style={
    overlay,
    append after command={
      ([xshift={+#1}]\tikzlastnode.north west)
      ([xshift={+-#1}]\tikzlastnode.south east)}},
  down and trim/.default=1cm,
  down and trim/.style={
    yshift=-(\pgfmatrixcurrentcolumn-1)*1.5\baselineskip,
    trim node={#1}},
  downup and trim/.default=1cm,
  downup and trim/.style={
    yshift=iseven(\pgfmatrixcurrentcolumn) ? -1.5\baselineskip : 0pt,
    trim node={#1}},
  -|/.style={to path={-|(\tikztotarget)\tikztonodes}},
  |-/.style={to path={|-(\tikztotarget)\tikztonodes}},
  -| sl/.style={-|, xslant=-1},
  |- sl/.style={|-, xslant= 1},
  center picture/.style={
    trim left=(current bounding box.center),
    trim right=(current bounding box.center)}}

\title{Anomalous Actions of Groups on Tensor Categories}
\author{Noah Lanier}
\date{\today}

\begin{document}

\maketitle

\thispagestyle{empty}
\begin{abstract} 
 For a group $G$ and a 4-cocycle $\pi\in Z^4(G,\kt)$, a $\pi$-anomalous action of $G$ on a linear monoidal category $\mathcal{C}$ is a linear monoidal 2-functor between 3-groups $3\Gr(G,\pi)\rightarrow \Aut(\mathcal{C})$ where the latter denotes the 3-group of autoequivalences of $\mathcal{C}$. Given $G$ and $\pi$, we provide a method of constructing anomalous actions of $3\Gr(G,\pi)$ on a tensor categories.
\end{abstract}

\section{Introduction}
\
For a tensor category\footnote{i.e. linear, monoidal category} $\mathcal{C}$, the 2-group $\aut(\mathcal{C})$ is a standard way to capture symmetries of the category. However, if we think of $\mathcal{C}$ as a 2-category with one object, there is a higher group of symmetries $\Aut(\mathcal{C})$ which is a linear 3-group. Given a 4-cocycle $\pi\in Z^4(G,\kt)$, we can define a $\pi$\textit{-anomalous action} of $G$ on $\mathcal{C}$ as a 3-functor from 
$3\Gr(G,\pi)$ to $\Aut(\mathcal{C})$. Thus, tensor categories have a natural notion of anomalous symmetries, generalizing the idea of anomalous symmetries on algebras from \cite{Jones_2021}. The 4-cocycle $\pi$ appears as obstruction to problems in extension theory and topological phases of matter. A natural question to ask is: given a group and a 4-cocycle how can we construct $\pi$-anomalous actions on tensor categories?

From general theory, it is known that for any group $G$ and $\pi\in Z^4(G,\kt)$ there exists an anomalous action on some fusion category \cite{Johnson_Freyd_2023}. However, we want to give a recipe for a construction that produces explicit anomalous actions when given explicit algebraic data. In this paper, we are generalizing the techniques and methodologies from \cite{Jones_2021}, which studied this problem one categorical level down on \textit{associative algebras}\footnote{Jones' paper was written in terms of $C^*$-algebras, but only needed the associative structure}, to this new setting that will allow us to construct explicit anomalous actions on \textit{tensor categories}. We now introduce the main theorem:

\bigskip

\begin{theorem}\label{theorem1}
Suppose we have the following data:
\begin{itemize}
    \item A group $Q$ and $[\pi]\in H^4(Q,\kt)$, with a normalized representative $\pi\in Z^4(Q,\kt)$.
    \item A group $G$ and a surjective homomorphism $\rho:G\rightarrow Q$ with kernel $K$.
    \item A normalized cochain $\omega\in C^3(G,\kt)$ such that $d\omega=\rho^*(\pi)$.
    \item An action of $G$ on the tensor category $\B$.
\end{itemize}
Then there exists a $\pi$-anomalous action of $Q$ on the twisted crossed product tensor category $\B\rtimes_{\omega} K$ where $\omega\in Z^3(K,\kt)$ is the restriction of $\omega$ to $K$.    
\end{theorem}

\bigskip

In the literature of extension theory of fusion categories, obstruction theory, gauging of fusion categories and more, there exists a class $o_4\in H^4(G,\kt)$ known as an ($H^4$-)obstruction \cite{Barkeshli_2019,cmp/1104180750,etingof2009fusioncategorieshomotopytheory,Kong_2020,Thorngren_2015,PhysRevX.8.031048}. The typical question to ask in these areas is whether or not the obstruction vanishes. In \cite{etingof2009fusioncategorieshomotopytheory}, $G$-extensions of fusion categories $\mathcal{C}$ are parametrized by 3-homomorphisms\footnote{also known as trihomomorphisms, 3-group homomorphisms, 3-functors, etc.} $\underline{\underline{\rho}}:\underline{\underline{G}}\rightarrow \underline{\underline{\brpic}}(\mathcal{C})$. The question of whether such a 3-homomorphism exists, starting with a group homomorphism $\rho:G\rightarrow \brpic(\mathcal{C})$, are questions whether $\rho$ lifts to a 2-homomorphism $\underline{\rho}$ and whether the 2-homomorphism $\underline{\rho}$ (if it exists) lifts to a 3-homomorphism $\underline{\underline{\rho}}$. The questions of whether the homomorphisms lift are the same as asking whether obstructions $o_3(\rho)\in H^3(G,\text{Aut}(1_{\mathcal{C}}))$ and $o_4(\underline{\rho})\in H^4(G,\kt)$ vanish. 

Even in simple settings, the obstruction $o_4$ is notoriously hard to compute (see \cite{Cui_2016} for a formula in a certain setting). 
We want to reverse these problems and instead of considering $o_4$ an obstruction we will start with $[\pi]\in H^4(G,\kt)$, consider it as extra data and refer to it as an anomaly. Similarly to the classical extension theory for fusion categories, we want to build 3-homomorphisms, but now using $\pi$ as extra data, we want 3-homomorphisms $3\Gr(G,\pi)\rightarrow \Aut(\mathcal{C})(\hookrightarrow\underline{\underline{\brpic}}(\mathcal{C}))$ where $3\Gr(G,\pi)$ is the 3-group\footnote{also known as categorical 2-group} similar to $\underline{\underline{G}}$, now where the pentagon equation is only satisfied up to $\pi$ and is linear in the sense that the top level of morphisms form a torsor over $\kt$. 

In \cite{galindo2017categoricalfermionicactionsminimal}, it was shown that for a braided fusion category $\mathcal{B}$ whose Müger (or symmetric) center is \textit{Tannakian} (i.e. $\mathcal{Z}_2(\mathcal{B})\cong \text{Rep}(G)$ for some finite group $G$) the obstruction to emitting a minimal nondegenerate extension lives in $H^4(G,\kt)$. Moreover, it is known that every class in $H^4(G,\kt)$ can be realized as such an obstruction for some braided fusion category \cite{Johnson_Freyd_2023}.
In (2+1)-D topological phases of matter, extending fusion categories by (finite) groups $G$ (what are called $G$\textit{-crossed braided fusion categories}) is an integral step in the process of gauging, which describes  the process of promoting a (finite) group of global symmetries to local symmetries \cite{Barkeshli_2019,Cui_2016,Gannon_2019}. Thus, the obstruction $o_4$ to $G$-extensions of fusion categories may also be seen as an obstruction to gauging symmetries of topological phases of matter. 

If one is interested in understanding topological (or gapped) boundaries of Bosonic (3+1)-D topological orders there is a general correspondence between such topological orders and Morita equivalences of two-algebra objects in $2\text{-}\vek(G,\pi)$ \cite{Lan_2018}. We also have a general correspondence between two-algebra objects in $2\text{-}\vek(G,\pi)$ and 3-homomorphisms  $3\Gr(G,\pi)\rightarrow \underline{\underline{\brpic}}(\mathcal{C})$ where $\mathcal{C}$ is fusion \cite{douglas2018fusion2categoriesstatesuminvariant,décoppet2021weakfusion2categories,D_coppet_2024,décoppet2024classificationfusion2categories}.
Our recipe for constructions of 3-homomorphisms should thus help produce examples of topological boundaries of (3+1)-D topological orders.

\textbf{Outline.} In section 2 we review preliminary definitions and introduce terminology and notation that will be used throughout. Section 3 is devoted to proving the main theorem (\ref{theorem1}.) We will separate the work into a series of lemmas. We note that, because we are working with higher categories, that the diagrams that we need to commute can be quite large and so, when we find it may be helpful, we will use a color-coding for such diagrams. Finally, section 4 is where we investigate some concrete examples of anomalous actions and give guidelines for some convenient methods of producing more examples in the future. 

\textbf{Acknowledgments.} I would like to thank Corey Jones for the many conversations had in the completion of this paper. The main result is essentially a categorification of theorem 3.1 in \cite{Jones_2021} which was inspired by the work of Vaughan Jones. I never had the pleasure of meeting Vaughan, though I am glad an extension of his ideas lives on in this work. This research was supported by NSF grant DMS-2247202.

\newpage
\section{Preliminaries}
We will assume some level of knowledge of tensor categories. For a comprehensive overview of tensor categories see \cite{etingof2016tensor}. The following section is primarily for setting conventions used throughout the paper.

\subsection{2-groups}

Our monoidal categories will be \textit{strictly unital}, which means $1_{\mathcal{C}}\otimes a=a\otimes 1_{\mathcal{C}}=a$ for all $a\in \mathcal{C}$ and $l_a=\id_a=r_a$. 
\begin{definition}
    A \textbf{2-group} (or categorical-group) $\mathcal{G}$ is a monoidal category where every object has a weak inverse and every morphism is an isomorphism.
\end{definition}

\begin{exmp}
    A simple example of a 2-group is the discrete 2-group\footnote{This is often referred to as $BG$, the \textit{delooping} of $G$ \cite{lurie2008highertopostheory}} $\underline{G}$ associated to a group $G$. The objects of $\underline{G}$ are the elements of $G$, the arrows are only the identities, and the tensor product is the operation in $G$.
\end{exmp}

\begin{exmp}
    Another important example of a 2-group that we will see in this paper is the 2-group $\aut(\mathcal{C})$ for a given tensor category $\mathcal{C}$. The objects are given by monoidal autoequivalences of $\mathcal{C}$ and the morphisms are natural isomorphisms. The tensor product is given by composition of functors, the tensor unit is the identity functor $1_{\mathcal{C}}$ and the associator and unitors are trivial.
\end{exmp}

\subsection{Group cohomology}
Group cohomology plays an important role in higher category theory \cite{baez1997introductionncategories,baez2007lecturesncategoriescohomology}. We will include general definitions
to set conventions. For groups of functions valued in
$\kt$, we will use multiplicative notation.

Let $G$ be a group and like $M$ be a $G$-module. Define
\[C^k(G,M):=\{f:G^{\times k}\rightarrow M\,|\,f(g_1,\cdots,g_k)=0\text{ if any }g_i=1\}.\]
The elements of $C^k(G,M)$ are called \textit{normalized} cochains. Define the differential\footnote{This definition of $d$ may differ by a sign from other conventions} $d^k:C^k(G,M)\rightarrow C^{k+1}(G,M)$ by
\begin{align*}
    d^kf(g_1,\cdots,g_{k+1})&=g_1(f(g_2,\cdots,g_{k+1}))\\
    &+\sum_{i=1}^k(-1)^{i}f(g_1,\cdots,g_ig_{i+1},\cdots,g_{k+1})\\
    &+(-1)^{k+1}f(g_1,\cdots,g_k).
\end{align*}
We note that $d^{k+1}\circ d^k=0$. We often drop mention of the superscript, assuming it is understood from context. Then set $Z^k(G,M):=\ker(d^k)$, $B^k(G,M):=\text{Im}(d^{k-1})$ and $H^k(G,M):=Z^k(G,M)/B^k(G,M)$. The $Z^k$ are called \textit{normalized cocycles}, the $B^k(G,M)$ are called \textit{normalized coboundaries} and $H^k(G,M)$ is called the $k^{\text{th}}$ cohomology group of $G$ with coefficients in $M$. 

A \textit{cochain complex} is defined to be the chain
\[\cdots\xrightarrow{d} C^{k-1}(G,M)\xrightarrow{d}C^k(G,M)\xrightarrow{d} C^{k+1}(G,M)\xrightarrow{d} \cdots\]
The un-normalized versions of the above are defined almost precisely the same way, except the definition of
the group of cochains removes the condition that $f(g_1,\cdots,g_k)=0$ if $g_i=1$. The inclusion of the normalized chain complex into the un-normalized version induces an isomorphism on cohomology
groups. The equivalence between these approaches is detailed in (\cite{weibel1994introduction}, Section 6.5). The normalized version of group cohomology is certainly more convenient for applications “in nature”, e.g. in
group extensions and higher category theory. When we say cocycle or coboundary in the paper,
we will always mean a normalized cocycle or coboundary. 

Group cohomology will play an essential role in our story, especially cohomology for finite
groups. Standard references for group cohomology include \cite{brown2012cohomology,weibel1994introduction}. Given a homomorphism $\rho:G\rightarrow H$, then any $G$-module $M$ is endowed with the structure of an $H$-module via $h\cdot m:=\rho(h)\cdot m$. The $\rho$ induces a homomorphism of cochain complexes called the \textit{pullback} \[\rho^*:C^*(H,M)\rightarrow C^*(G,M),\]
\[\rho^*(f)(h_1,\cdots,h_k):=f(\rho(h_1),\cdots,\rho(h_k)).\]
The pullback is a morphism of cochain complexes, hence induce homomorpisms at the level of cocycles, coboundaries and cohomology.

Given a short exact sequence of groups
\[0\rightarrow L\rightarrow M\rightarrow N\rightarrow 0,\]
we have the \textit{long exact sequence in cohomology} 
\[\cdots\rightarrow H^{k-1}(G,N)\rightarrow H^{k}(G,L)\rightarrow H^k(G,M)\rightarrow H^k(G,N)\rightarrow H^{k+1}(G,L)\rightarrow \cdots\]
where the arrows between cohomology groups of the same degree are pushforwards. The degree shifting maps $H^{k-1}(G,N)\rightarrow H^k(G,L)$ are called \textit{connecting homomorphisms}.

A \textit{cup product} is a map $H^k(G,\F)\times H^m(G,\F)\rightarrow H^{k+m}(G,\F)$ defined by $f\in C^k(G,\F)$, $g\in C^m(G,\F)$
\[fg(x_1,\dots,x_{k+m}):=f(x_1,\dots,x_k)\cdot g(x_{k+1},\dots,x_{k+m}).\]
Note that the operation $\cdot$ above comes from the ring structure of $\F$ and therefore in order for a cup product to be well-defined $\F$ must have a ring structure. Cup products will be helpful in the construction of explicit examples of anomalous actions.

\newpage
\subsection{Group actions on tensor categories}
For group actions on tensor categories we will be using some combination of notation from \cite{bernaschini2017groupactions2categories,galindo2015crossed,tambara2001semi}. An important convention to single out is that our associators will associate right to left. We include definitions below for reference of notational conventions. 
\begin{definition}
    An action of a group $G$ on a tensor category $\B$ consists of data
    \begin{itemize}
        \item Monoidal autoequivalences $(\sigma_*,\psi^{\sigma}):\B\rightarrow\B$ for all $\sigma\in G$,
        \item nautral isomorphisms $\chi_{\sigma,\tau}:\sigma_*\circ\tau_*\rightarrow(\sigma\tau)_*$,
        \item an isomorphism $\nu:e_*\rightarrow\id_{\B}$,
    \end{itemize}
    which makes the following diagram commute for all $\sigma,\tau,\rho\in G$ and $V\in \B$. 
\begin{equation}\label{eq1.1}
\begin{tikzcd}
	{\sigma_*\tau_*\rho_* V} && {(\sigma\tau)_*\rho_* V} \\
	\\
	{\sigma_*(\tau\rho)_* V} && {(\sigma\tau\rho)_* V}
	\arrow["{\sigma_*(\chi_{\tau,\rho}(V))}"', from=1-1, to=3-1]
	\arrow["{\chi_{\sigma,\tau\rho}(V)}"', from=3-1, to=3-3]
	\arrow["{\chi_{\sigma\tau,\rho}(V)}", from=1-3, to=3-3]
	\arrow["{\chi_{\sigma,\tau}(\rho_* V)}", from=1-1, to=1-3]
\end{tikzcd}    
\end{equation}
By definition of a monoidal functor, the above $\sigma_*$ consist of 
\begin{itemize}
    \item a functor $\sigma_*:\B\rightarrow\B$,
    \item natural isomorphisms $\psi^\sigma_{V,W}:\sigma_* V\otimes \sigma_* W\rightarrow\sigma_*(V\otimes W)$ for all $V,W\in \B$,
\end{itemize}
making the following diagram commutative for all $V,W,Z\in \B$.
\begin{equation}\label{eq1.2}
\begin{tikzcd}
	{\sigma_* V\otimes(\sigma_* W\otimes \sigma_* Z)} && {(\sigma_* V\otimes \sigma_* W)\otimes \sigma_* Z} \\
	\\
	{\sigma_* V\otimes\sigma_*(W\otimes Z)} && {\sigma_*(V\otimes W)\otimes \sigma_* Z} \\
	\\
	{\sigma_*(V\otimes(W\otimes Z))} && {\sigma_*((V\otimes W)\otimes Z)}
	\arrow["{\alpha_{\sigma_* V,\sigma_* W,\sigma_* Z}}", from=1-1, to=1-3]
	\arrow["{\text{id}_{\sigma_* V}\otimes\psi^{\sigma}_{W,Z}}"', from=1-1, to=3-1]
	\arrow["{\psi^{\sigma}_{V,W\otimes Z}}"', from=3-1, to=5-1]
	\arrow["{\sigma_*(\alpha_{V,W,Z})}"', from=5-1, to=5-3]
	\arrow["{\psi^{\sigma}_{V,W}\otimes\text{id}_{\sigma_* Z}}", from=1-3, to=3-3]
	\arrow["{\psi^{\sigma}_{V\otimes W, Z}}", from=3-3, to=5-3]
\end{tikzcd}
\end{equation}
The requirement that $\chi_{\sigma,\tau}$ is a transformation of monoidal functors means that the following diagram is commutative for all $V,W\in \B$.
\begin{equation}\label{eq1.3}
    \begin{tikzcd}
	{\sigma_*\tau_* V\otimes \sigma_*\tau_* W} &&& {(\sigma\tau)_* V\otimes (\sigma\tau)_* W} \\
	{\sigma_*(\tau_*V\otimes\tau_* W)} \\
	{\sigma_*\tau_*(V\otimes W)} &&& {(\sigma\tau)_*(V\otimes W)}
	\arrow["{\psi^{\sigma\tau}_{V,W}}", from=1-4, to=3-4]
	\arrow["{\chi_{\sigma,\tau}(V)\otimes\chi_{\sigma,\tau}(W)}", from=1-1, to=1-4]
	\arrow["{\psi^{\sigma}_{\tau_* V,\tau_* W}}"', from=1-1, to=2-1]
	\arrow["{\sigma_*(\psi^{\tau}_{V,W})}"', from=2-1, to=3-1]
	\arrow["{\chi_{\sigma,\tau}(V\otimes W)}"', from=3-1, to=3-4]
\end{tikzcd}
\end{equation}

We will not concern ourselves with the left and right unitor diagrams as our unitors are always trivial. It is often convenient to refer to a tensor category with $G$ action as a \textit{$G$-tensor category}. In other words, an action of $G$ on a tensor category $\B$ is the same data as a monoidal functor from the 2-group $\underline{G}$ to the 2-group $\aut(\B)$
\[-_*:\underline{G}\rightarrow\aut(\B).\]
\end{definition}
\begin{definition}
   Given a group $G$ and a $G$-tensor category $\B$, we shall define a $\bm{G}$\textbf{-crossed product tensor category} associated to this action, denoted $\B\rtimes G$. We set $\B\rtimes G=\bigoplus_\sigma \B_\sigma$ as a semisimple category, where $\B_\sigma=\B$. We denote by $[V,\sigma]$ the object $V\in \B_\sigma$ and a morphism from $[V_\sigma,\sigma]$ to $[W_\sigma,\sigma]$ is expressed as $[f_\sigma,\sigma]$ with $f_\sigma:V_\sigma\rightarrow W_\sigma$ a morphism in $\B$.

The tensor product $\cdot:\B\rtimes G \times \B\rtimes G\rightarrow \B\rtimes G$ is defined by
\begin{equation}\label{eq1.4}
\begin{split}
    [V,\sigma]\cdot [W,\tau]&=[V\otimes \sigma_* W,\sigma\tau] \quad \text{for objects, and}\\
    [f,\sigma]\cdot [g,\tau]&=[f\otimes \sigma_* g,\sigma\tau]\quad \text{for morphisms.}
\end{split}
\end{equation}
The unit object is $(1_\B,e)$. The associativity is given by 
\begin{equation}\label{eq1.5}
\begin{tikzcd}
	{[V,\sigma]\cdot([W,\tau]\cdot [Z,\rho])} & {[V,\sigma]\cdot[W\otimes\tau_* Z,\tau\rho]} & {[V\otimes(\sigma_*(W\otimes\tau_*Z)),\sigma\tau\rho]} \\
	\\
	{([V,\sigma]\cdot[W,\tau])\cdot [Z,\rho]} & {[V\otimes \sigma_* W,\sigma\tau]\cdot [Z,\rho]} & {[(V\otimes \sigma_*W)\otimes(\sigma\tau)_*Z,\sigma\tau\rho]}
	\arrow["{\alpha_{[V,\sigma],[W,\tau],[Z,\rho]}}"', from=1-1, to=3-1]
	\arrow["{[\alpha^{\mathcal{C}\rtimes G}(V,\sigma,W,\tau,Z,\rho),\sigma\tau\rho]}"', from=1-3, to=3-3]
	\arrow["{=}"{description}, draw=none, from=1-1, to=1-2]
	\arrow["{=}"{description}, draw=none, from=1-2, to=1-3]
	\arrow["{=}"{description}, draw=none, from=3-1, to=3-2]
	\arrow["{=}"{description}, draw=none, from=3-2, to=3-3]
\end{tikzcd}
\end{equation}
where $\alpha^{\B\rtimes G}(V,\sigma,W,\tau,Z,\rho)$ is the composite
\begin{equation}\label{eq1.6}
\begin{tikzcd}
	{V\otimes\sigma_*(W\otimes\tau_* Z)} \\
	\\
	{V\otimes(\sigma_*W\otimes\sigma_*\tau_* Z)} \\
	\\
	{V\otimes(\sigma_*W\otimes(\sigma\tau)_* Z)} \\
	\\
	{(V\otimes\sigma_*W)\otimes(\sigma\tau)_* Z}
	\arrow["{\text{id}_V\otimes (\psi^{\sigma_*}_{W,\tau_* Z})\inv}"', from=1-1, to=3-1]
	\arrow["{\text{id}_{V\otimes\sigma_* W}\otimes\chi_{\sigma,\tau}}"', from=3-1, to=5-1]
	\arrow["{\alpha^{\mathcal{C}}(V,\sigma_* W,(\sigma\tau)_* Z)}"', from=5-1, to=7-1]
\end{tikzcd}
\end{equation}
Therefore 
\begin{equation}\label{eq1.7}
    \alpha_{[V,\sigma],[W,\tau],[Z,\rho]}=[\alpha^{\B}(V,\sigma_* W,(\sigma\tau)_* Z),\sigma\tau\rho]\circ [\text{id}_{V\otimes\sigma_* W}\otimes\chi_{\sigma,\tau},\sigma\tau\rho]\circ [\text{id}_V\otimes \psi^{\sigma_*}_{W,\tau_* Z},\sigma\tau\rho].
\end{equation}

The associativity constraint is given by the commutativity of the following diagram (see \cite{galindo2015crossed},\cite{tambara2001semi})
\[\begin{tikzcd}[scale cd=0.95,column sep = 0.75em, center picture]
	{[V,\sigma]\cdot([W,\tau]\cdot([Z,\rho]\cdot[U,\gamma]))} && {([V,\sigma]\cdot[W,\tau])\cdot([Z,\rho]\cdot[U,\gamma])} \\
	\\
	{[V,\sigma]\cdot(([W,\tau]\cdot[Z,\rho])\cdot[U,\gamma])} && {(([V,\sigma]\cdot[W,\tau])\cdot[Z,\rho])\cdot[U,\gamma]} \\
	\\
	& {([V,\sigma]\cdot([W,\tau]\cdot[Z,\rho]))\cdot[U,\gamma]}
	\arrow["{\text{id}_{[V,\sigma]}\cdot\alpha_{[W,\tau],[Z,\rho],[U,\gamma]}}"', from=1-1, to=3-1]
	\arrow["{\alpha_{[V,\sigma],[W,\tau],[Z,\rho]\cdot[U,\gamma]}}", from=1-1, to=1-3]
	\arrow["{\alpha_{[V,\sigma]\cdot[W,\tau],[Z,\rho],[U,\gamma]}}", from=1-3, to=3-3]
	\arrow["{\alpha_{[V,\sigma],[W,\tau]\cdot[Z,\rho],[U,\gamma]}}"', from=3-1, to=5-2]
	\arrow["{\alpha_{[V,\sigma],[W,\tau],[Z,\rho]}\cdot\text{id}_{[U,\gamma]}}"', from=5-2, to=3-3]
\end{tikzcd}\]

\begin{remark}\label{rem1}
    From now on we shall denote $[V]:=[V,e]$ and $[\sigma]:=[I,\sigma]$, for all $V\in \B$ and $\sigma\in G$ (similarly for morphisms). Note that $[V]\cdot [\sigma]=[V,\sigma]$ and $[\sigma]\cdot[V]=[\sigma_*V,\sigma]$ for all $V\in \B$ and $\sigma\in G$. We also note that every object of $\B\rtimes G$ is a direct sum of tensor products of $[V],[\sigma]$ for $V\in \B$ and $\sigma\in G$ and every morphism is a direct sum of tensor products of $[f],[\sigma]$ for some $f:V\rightarrow W$ in $\B$ and $\sigma\in G$.
\end{remark}

For the pentagonal identity of $\B\rtimes G$ see \cite{galindo2015crossed}. 
\end{definition}

\bigskip

\begin{definition}
    We define a \textbf{twisted} $G$-crossed product on $\B$ as a $G$-crossed product tensor category along with a (normalized) 3-cocycle, $\omega\in Z^3(G,\kt)$, where 
    \[\alpha^{\omega}_{[V,\sigma],[W,\tau],[Z,\rho]}=\omega(\sigma,\tau,\rho)\alpha_{[V,\sigma],[W,\tau],[Z,\rho]}\]
    for all $\sigma,\tau,\rho\in G$.

    We use the notation $\B\rtimes_\omega G$ to denote the twist by $\omega$. The 3-cocycle condition is equivalent to the associativity constraint and the condition of normality ensures that $\alpha^{\omega}_{[V,e],[W,\tau],[Z,\rho]}=\id$ and so the pentagonal identity for $\B\rtimes_\omega G$ follows similarly to \cite{galindo2015crossed}. 
\end{definition}

\newpage
\subsection{$\kt$-Linear 3-groups } 
We will use the notation ``$n$-morphisms'' to refer to the objects/morphisms at different categorical levels. Some use the terminology ``$n$-cells'' for the same purpose.
\begin{definition}
    A \textbf{$\kt$-linear 3-group} (or $\kt$-linear categorical 2-group) is a monoidal 2-category (for an in-depth review of monoidal 2-categories see \cite{ahmadi2023monoidal}) in which every object has a weak inverse under the tensor product, every 1-morphism has a weak inverse, and 2-morphisms form a torsor over $\kt$ and composition of 2-morphisms intertwines the $\kt$ action, i.e.,
    \[af\circ g=f\circ ag=a(f\circ g).\]

\end{definition}
\begin{remark}
    We note that our notion of \textit{linear} differs slightly from the notion of linear categories that have $\mathbbm{k}$-linear vector spaces as hom-spaces. The hope is that the notation $\kt$-linear creates enough of a distinction and helps excuse this sloppy terminology.
\end{remark}

\begin{exmp}
Given a $\mathbbm{k}$-linear monoidal category $(\B,\otimes,\alpha,1_\B)$ we define  $\Aut(\B)$ as follows: 
\begin{itemize}
    \item A singleton Obj$(\Aut(\B))=\bullet$ (0-morphism),
    \item A 2-category  $\Aut(\B)(\bullet,\bullet)$ with
    \begin{itemize}
        \item objects given by linear monoidal autoequivalences of $\B$ (1-morphisms),
        \item 1-morphisms given by linear pseudonatural isomorphisms of monoidal autoequivalences of $\B$ (2-morphisms),
        \item 2-morphisms form a $\kt$-torsor, given by invertible modifications with coefficients in $\kt$ (3-morphisms.)
    \end{itemize}
    
    \item Tensor product $\Bar{\otimes}$ for (1-3)-morphisms given by
    \begin{itemize}
        \item[1-morphisms:] (bilinear) composition of monoidal autoequivalences (written as juxtaposition),
        \item[2-morphisms:] If $(U_{K,K'},\chi'):K\Rightarrow K'$ and $(U_{H,H'},\chi):H\Rightarrow H'$ are pseudonatural isomorphisms, the (bilinear) tensor product $(U_{K,K'},\chi')\Bar{\otimes} (U_{H,H'},\chi):KH\Rightarrow K'H'$ is defined as
        \[(U_{K,K'},\chi')\Bar{\otimes} (U_{H,H'},\chi):=(K(U_{H,H'})\otimes U_{K,K'},\chi'\otimes \chi),\] where $\chi'\otimes \chi$ is given by the commutativity of the following diagram
\[\begin{tikzcd}[scale cd=0.85, center picture] 
	{KH(V)\otimes K(U_{H,H'})\otimes U_{K,K'}} && {K(U_{H,H'})\otimes U_{K,K'}\otimes K'H'(V)} \\
	\\
	{K(H(V)\otimes U_{H,H'})\otimes U_{K,K'}} & {K(U_{H,H'}\otimes H'(V))\otimes U_{K,K'}} & {K(U_{H,H'})\otimes KH'(V)\otimes U_{K,K'}}
	\arrow["{\psi^K\Bar{\otimes} \text{id}}"', from=3-1, to=1-1]
	\arrow["{K(\chi(V))\Bar{\otimes}\text{id}}", from=3-1, to=3-2]
	\arrow["{\psi^K\Bar{\otimes}\text{id}}", from=3-2, to=3-3]
	\arrow["{\text{id}\Bar{\otimes}\chi'(V)}", from=3-3, to=1-3]
	\arrow["{(\chi'\Bar{\otimes}\chi)(V)}", from=1-1, to=1-3]
\end{tikzcd}\]
where the subscripts of arrows are omitted for brevity.
        \item[3-morphisms:] For pseudonatural isomorphisms $\chi_1,\chi_2:K\Rightarrow K'$ and $\phi_1,\phi_2:H\Rightarrow H'$, and modifications $g:\chi_1\rightarrow\chi_2$ and $f:\phi_1\rightarrow\phi_2$ the tensor product of $g$ and $f$ is defined as $g\Bar{\otimes} f:=K(f)\otimes g$.
    \end{itemize}
\end{itemize}
The remaining properties of $\Aut(\B)$ as a $\kt$-linear 3-group follows intuitively, but will not be used in our construction and so we will omit the remaining data.
\newpage
\end{exmp}
\begin{exmp}
    
Given a triple $(G,A,\pi)$ where $G$ is a group, $A$ is a $G$-module and $\pi\in Z^4(G,A)$ is a (normalized) 4-cocycle we can build a (strictly unital) $A$-linear 3-group denoted $3\Gr(G,A,\pi)$ as follows (where similarly to above there is a singular 0-morphism):
\begin{itemize}
    \item The 1-morphisms are given by the elements of $G$,  2-morphisms between two 1-morphisms are given by
    \[\Hom(g,h)=\begin{cases} 
      \emptyset & \text{if }g\neq h  \\
      \id_g & \text{if }g=h, 
   \end{cases}\]
    and 3-morphisms are given by \[\Hom(\id_g,\id_g)=A\] 
   where we identify the identity in the category, $1_{\id_g}\in\Hom(\id_g,\id_g)$,  with the unit in $A$. Therefore we denote a 3-morphism $a\in \Hom(\id_g,\id_g)$ by $a1_{\id_g}$. 
   \item $g\otimes h:=gh$, the group unit $e$ is the monoidal unit.
   \item $\id_g\otimes\id_h:=\id_{gh}$ and $(a1_{\id_g})\otimes(b1_{\id_h}):=(ag(b))1_{\id_{gh}}$. \footnote{We will only consider $A$ with trivial $G$-action in this paper, therefore $(a1_{\id_g})\otimes(b1_{\id_h})=(ab)1_{\id_{gh}}$}
   \item The associator 
   \[\alpha_{g,h,j}:g\otimes (h\otimes j)=ghj\rightarrow (g\otimes h)\otimes j=ghj\]
   and unitor isomorphisms in the category are trivial.
   \item For every 4 objects $g,h,j,k,$ there is an invertible modification between compositions of associators called the pentagonator $\pi_{g,h,j,k}$ where 
   \[\pi_{g,h,j,k}:=\pi(g,h,j,k)1_{\id_{ghjk}}:\alpha_{gh,j,k}\circ \alpha_{g,h,jk}\rightarrow (\alpha_{g,h,j}\otimes \id_k) \circ \alpha_{g,hj,k}\circ (\id_g\otimes  \alpha_{h,j,k}).\]
\end{itemize}

\end{exmp}
\begin{remark}
    The 3-groups given by triples $(G,A,\pi)$ in this paper are only focused on the case that $A=\kt$ with trivial $G$ action and thus for brevity we will drop the $\kt$ and simple refer to such 3-groups as $3\Gr(G,\pi)$.
\end{remark}
The natural notion for a functor between $\kt$-linear 3-groups we will refer to as a \textbf{$\kt$-linear monoidal 2-functor}. We refer the reader to \cite{gurski2013monoidal} for more details on monoidal 2-functors. 

\begin{definition}
    A (strictly unital) \textbf{$\kt$-linear monoidal 2-functor} between (strictly unital) $\kt$-linear monoidal $2$-categories consists of
    \begin{itemize}
        \item A functor of the underlying\footnote{In this case by `underlying' we mean trivialize the 3-morphisms. In essence, we are cauterizing off the `top' layer of our 3-categories by taking equivalence classes of 2-morphisms up to invertible 3-morphisms.} monoidal categories $F:\B\rightarrow\B'$,
        \item Pseudonatural isomorphisms $(U_{\sigma,\tau},\chi_{\sigma,\tau}):F(\sigma)\otimes F(\tau)\Rightarrow F(\sigma\otimes \tau)$ for all $\sigma,\tau\in \B$ and
        \item Invertible modifications $\omega_{\sigma,\tau,\rho}:\chi_{\sigma,\tau\rho}\otimes (\id_{F(\sigma)}\otimes \chi_{\tau,\rho})\Rrightarrow\chi_{\sigma\tau,\rho}\otimes(\chi_{\sigma,\tau}\otimes\id_{F(\rho)})$ for all $\sigma,\tau,\rho\in \B$.
    \end{itemize}
    all subject to axioms which are identical to the trihomomorphism axioms (see \cite{gurskialgtri}) aside from source and target considerations.
\end{definition}

\newpage
\section{Anomalous actions}
\begin{definition} An \textbf{anomalous action} of a group $G$ on a $\mathbbm{k}$-linear monoidal category $\mathcal{C}$ consists of a $4$-cocycle $\pi\in Z^4(G,\kt)$ and a $\kt$-linear monoidal 2-functor $3\Gr(G,\pi)\rightarrow \Aut(\mathcal{C})$. The cohomology class $[\pi]\in H^4(G,\kt)$ is called the \textbf{anomaly}. If $\pi\in Z^4(G,\kt)$ is given it is convenient to use the terminology \textit{$\pi$-anomalous action}. This definition is a categorified version of the definition of an anomalous action of a group on an algebra in \cite{Jones_2021}.
\end{definition}
\begin{remark}
    In many applications of physics it is important that things are unitary and in that case we would just choose our field to be $\C$ and identify $\C^{\times}\cong U(1)$ in the natural way. We note that none of the structure of the arguments change in this case and cohomology valued in $U(1)$ is equal to cohomology valued in $\C^{\times}$.
\end{remark}

We want to know how to construct explicit anomalous actions and therefore the  goal of this section is to prove the main theorem of this paper, theorem (\ref{theorem1}). We do this in a series of lemmas. First we unpack some of the structure we are given. The equation $d\omega=\rho^*(\pi)$ unpacks to give, for all $x,y,w,z\in G$,
\begin{equation}\label{eq2.1}
\begin{split}
    \rho^*(\pi)(x,y,z,w) &=\frac{\omega(y,z,w)\omega(x,yz,w)\omega(x,y,z)}{\omega(xy,z,w)\omega(x,y,zw)}   \\ \omega(y,z,w)\omega(x,yz,w)\omega(x,y,z)&=\pi(\rho(x),\rho(y),\rho(z),\rho(w))\omega(xy,z,w)\omega(x,y,zw).
\end{split}
\end{equation}
Furthermore, since $\pi$ is assumed to be normalized, we have
\begin{equation}\label{eq2.2}
    \rho^*(\pi)(x,y,w,z)=1 \text{ if any }x,y,w,z\in K.
\end{equation}
Thus the restriction of $\omega$ to $K$ is indeed a (normalized) 3-cocycle. Restricting the $G$-action on $\B$ to $K$ then gives the twisted $K$-crossed product category $\BK$. Our $G$ action on $K$ is given by conjugation and as notation, we let 
\[k^{g}:=g kg\inv\quad\text{for all } k\in K \text{ and } g\in G.\]
The goal will be to use the given data to define a $\kt$-linear monoidal 2-functor 
\[3\Gr(Q,\pi)\rightarrow \Aut(\BK).\]
To preform this construction, we first pick a set theoretic (unital) section of $\rho$ where $q\mapsto \q$ for all $q\in Q$. 

\begin{definition}
For each $q\in Q$, we define an endofunctor of $\BK$ as
\begin{equation}\label{eq2.8}
\begin{split}
    q_*&:\BK\rightarrow\BK\\
    &\,[V,k]\mapsto [\q_* V,k^{\q}]\\
    &\,[f,k]\mapsto [\q_* f,k^{\q}].
\end{split}
\end{equation} 
which is well-defined by the $G$-action on $\B$. We also define maps 
\[\psit^q_{[V,k],[W,l]}:=\left(\frac{\omega(kq\inv,q,lq\inv)\omega(k,l,q\inv)}{\omega(q,kq\inv,l^{q})\omega(k,q\inv,q)}\right)([\psi^{\q}_{V,k_* W},(kl)^{\q})]\circ[\id\otimes\chi_{\q,k}(W)\inv,(kl)^{\q}]\circ[\id \otimes\chi_{k^{\q},\q}(W),(kl)^{\q}])\]
\[\psit^q_{[V,k],[W,l]}:q_*[V,k]\cdot q_*[W,l]\rightarrow q_*([V,k]\cdot[W,l])\]
for every $q\in Q$.
\end{definition}

\newpage

\begin{lemma}
    For any $q\in Q$, $(q_*,\psit^q)$ may be equipped with the structure of a linear monoidal autoequivalence of $\BK$. 
\end{lemma}

\begin{proof}
    It is clear that $q_*$ defines a functor from $\BK$ to itself. 
    By definition, $\psit^q$ is a composition of natural isomorphisms (times a scalar) 
    \[\psitq_{[V,k],[W,l]}=\left(\frac{\omega(kq\inv,q,lq\inv)\omega(k,l,q\inv)}{\omega(q,kq\inv,l^{q})\omega(k,q\inv,q)}\right)[\psi^{\q}_{V,k_* W},(kl)^{\q})]\circ[\id\otimes\chi_{\q,k}(W)\inv,(kl)^{\q}]\circ[\id \otimes\chi_{k^{\q},\q}(W),(kl)^{\q}]\]
    We note that the isomorphisms $\chi_{k^{\q},\q}(W)$, $(\chi_{\q,k}(W))\inv$ and $\psi^{\q}_{V,k_* W}$ are defined by the $G$-action over $\B$. We need the following diagram to commute in order for $q_*$ to be monoidal
\begin{equation}\label{monoidal comm diagram for q}
    \begin{tikzcd}
	{q_*[V,k]\cdot (q_*[W,l]\cdot q_*[Z,m])} && {(q_*[V,k]\cdot q_*[W,l])\cdot q_*[Z,m]} \\
	{q_*[V,k]\cdot q_*([W,l]\cdot[Z,m])} && {q_*([V,k]\cdot [W,l])\cdot q_*[Z,m]} \\
	{q_*([V,k]\cdot([W,l]\cdot[Z,m]))} && {q_*(([V,k]\cdot[W,l])\cdot[Z,m])}
	\arrow["{\text{id}\otimes\psit^q}"', from=1-1, to=2-1]
	\arrow["{\psit^q}"', from=2-1, to=3-1]
	\arrow["\alpha^\omega", from=1-1, to=1-3]
	\arrow["{q_*(\alpha^\omega)}"', from=3-1, to=3-3]
	\arrow["{\psit^q\otimes \text{id}}", from=1-3, to=2-3]
	\arrow["{\psit^q}", from=2-3, to=3-3]
\end{tikzcd}
\end{equation}
where subscripts are omitted. In order to check the commutativity of this diagram, we must expand our $\alpha^\omega$'s and $\psitq$'s. First we compute 
\[q_*[V,k]\cdot (q_*[W,l]\cdot q_*[Z,m])=[\q_* V\otimes k^{\q}_*(\q_* W\otimes l^{\q}_* \q_* Z),(klm)^{\q}]\]

We note that each ``$K$-component'' will be $(klm)^{q}$, and thus as a space-saving measure we only consider the ``$\B$-component'' as well as omitting subscripts and coloring morphisms with the key
\[\text{morphisms with }\begin{cases} 
      \chi & \textcolor{myblue}{\text{blue}} \\
      \psi & \textcolor{myred}{\text{red}} \\
      \alpha^{\B} & \textcolor{mygreen}{\text{green}} .
   \end{cases}\]
We also suppress the tensor product $\otimes$, the hats on $q$ and refer to the identity morphism as $1$ as a space saving measure.

\newgeometry{noheadfoot=true,top=1cm,bottom=1cm}

\[\begin{tikzcd}[ampersand replacement=\&,scale cd=0.48,column sep = 0.95em, center picture] 
	\&\& {q_* Vk^q_*(q_* W l^q_* q_* Z)} \& {q_* V(k^q_* q_* W k^q_* l^q_* q_* Z)} \& {q_* V(k^q_* q_* W (kl)^q_* q_* Z)} \& {(q_* V k^q_* q_* W) (kl)^q_* q_* Z} \\
	\\
	\& {q_* V k^q_*(q_* W (ql)_* Z)} \&\& {q_* V(k^q_* q_* W k^q_* (ql)_* Z)} \&\& {q_* V ((qk)_* W (kl)^q_* q_* Z)} \& {(q_* V (qk)_* W) (kl)^q_* q_* Z} \\
	\\
	\& {q_* V k^q_*(q_* W q_*l_* Z)} \& {q_* V (k^q_*q_* W k^q_* q_*l_* Z)} \&\& {q_* V ((qk)_* W (qkl)_* Z)} \& {q_* V (q_* k_* W (kl)^q_* q_* Z)} \& {(q_* V (qk)_* W) (kl)^q_* q_* Z} \\
	\\
	{q_* V k^q_*q_*( W l_* Z)} \& {q_* V k^q_*(q_* W  q_*l_* Z)} \& {\fbox{\text{P.I}}} \& {q_* V ((qk)_* W (qk)_*l_* Z)} \& {q_* V ((qk)_* W(qk)_*l_* Z)} \& {(q_* V (qk)_* W) (qkl)_* Z} \& {(q_* V q_* k_* W) (kl)^q_* q_* Z} \& {q_*( V k_* W) (kl)^q_* q_* Z} \\
	\\
	\& {q_* V k^q_*q_*( W l_* Z)} \&\& {q_* V (qk)_*(W l_* Z)} \\
	{q_* V (qk)_*( W l_* Z)} \&\&\& {\fbox{\text{P.II}}} \& {q_* V  (q_* k_* W q_* k_* l_* Z)} \& {(q_* V (qk)_* W) (qk)_*l_* Z} \&\& {(q_* V q_* k_* W) (qkl)_* Z} \& {q_*( V k_* W) (qkl)_* Z} \\
	\& {q_* V q_* k_*( W l_* Z)} \\
	\&\& {q_* V q_* (k_* W k_* l_* Z)} \&\&\&\& {(q_* V q_* k_* W) (qkl)_* Z} \& {(q_* V q_* k_* W) q_*(kl)_* Z} \\
	\&\&\& {q_* V  (q_* k_* W q_* k_* l_* Z)} \\
	\& {q_* V q_*(kq^{-1})_*q_*( W l_* Z)} \&\&\& {(q_* V  q_* k_* W) q_* k_* l_* Z} \& {(q_* V q_* k_* W) (qk)_*l_* Z} \\
	\&\&\& {\fbox{H}} \&\&\& {q_*( V k_* W) (qkl)_* Z} \\
	\&\&\&\&\& {q_*( V  k_* W) (qk)_*l_* Z} \\
	\&\& {q_* (V  k_*( W l_* Z))} \&\& {q_*(V   k_* W) q_* k_* l_* Z} \\
	\& {q_* V q_* k_*( W l_* Z)} \\
	\&\&\&\& {q_*(( V  k_* W) k_*l_* Z)} \&\& {q_*( V k_* W) q_*(kl)_* Z} \\
	\& {q_* (V  k_*( W l_* Z))} \& {q_* (V (kq^{-1})_*q_*( W l_* Z))} \& {q_* (V  (k_* W k_* l_* Z))} \&\& {q_*( V  k_* W) q_* k_*l_* Z} \\
	\\
	\&\&\&\& {q_* (V  (k_* W (kl)_* Z))} \\
	\\
	\&\&\& {q_* (V  k_*( W l_* Z))} \&\&\& {q_*( V k_* W) q_*(kl)_* Z} \\
	\&\& {q_* (V  (k_* W k_* l_* Z))} \&\& {q_* ((V  k_* W) k_* l_* Z)} \\
	\\
	\&\&\& {q_* (V  (k_* W (kl)_* Z))} \&\& {q_*(( V k_* W) (kl)_* Z)}
	\arrow["{1 k^q_*(1 \chi)}"', color={rgb,255:red,59;green,62;blue,247}, from=1-3, to=3-2]
	\arrow["{1\psi^{-1}}", color={rgb,255:red,214;green,92;blue,92}, from=1-3, to=1-4]
	\arrow["{1\psi^{-1}}", color={rgb,255:red,214;green,92;blue,92}, from=3-2, to=3-4]
	\arrow["{1 (1 k^q_*(\chi))}"{description}, color={rgb,255:red,59;green,62;blue,247}, from=1-4, to=3-4]
	\arrow["{1(1\chi)}", shift left, color={rgb,255:red,59;green,62;blue,247}, from=1-4, to=1-5]
	\arrow["{1 (\chi1)}"{description}, color={rgb,255:red,59;green,62;blue,247}, from=1-5, to=3-6]
	\arrow["{\alpha^{\mathcal{C}}}", color={rgb,255:red,92;green,214;blue,92}, from=1-5, to=1-6]
	\arrow["{\alpha^{\mathcal{C}}}", color={rgb,255:red,92;green,214;blue,92}, from=3-6, to=3-7]
	\arrow["{(1\chi)1}", color={rgb,255:red,59;green,62;blue,247}, from=1-6, to=3-7]
	\arrow["{1k^q_*(1 \chi^{-1})}"', color={rgb,255:red,59;green,62;blue,247}, from=3-2, to=5-2]
	\arrow["{1\psi^{-1}}", color={rgb,255:red,214;green,92;blue,92}, from=5-2, to=5-3]
	\arrow["{1(1 k^q_*(\chi^{-1}))}", color={rgb,255:red,59;green,62;blue,247}, from=3-4, to=5-3]
	\arrow["{1 k_*^q(\psi)}"', color={rgb,255:red,214;green,92;blue,92}, from=5-2, to=7-1]
	\arrow["{1 \psi}", color={rgb,255:red,214;green,92;blue,92}, from=5-3, to=7-2]
	\arrow["{1k_*^q(\psi)}"', color={rgb,255:red,214;green,92;blue,92}, from=7-2, to=7-1]
	\arrow["{(1\chi^{-1})1}", color={rgb,255:red,59;green,62;blue,247}, from=3-7, to=5-7]
	\arrow["{1(\chi^{-1}1)}", color={rgb,255:red,59;green,62;blue,247}, from=3-6, to=5-6]
	\arrow["{\alpha^{\mathcal{C}}}", color={rgb,255:red,92;green,214;blue,92}, from=5-6, to=5-7]
	\arrow["\psi1", color={rgb,255:red,214;green,92;blue,92}, from=5-7, to=7-8]
	\arrow["{\alpha^{\mathcal{C}}}", color={rgb,255:red,92;green,214;blue,92}, from=5-6, to=7-7]
	\arrow["\psi1", color={rgb,255:red,214;green,92;blue,92}, from=7-7, to=7-8]
	\arrow["1\chi", color={rgb,255:red,59;green,62;blue,247}, from=7-7, to=10-8]
	\arrow["\psi1", color={rgb,255:red,214;green,92;blue,92}, from=10-8, to=10-9]
	\arrow["{1\chi^{-1}}", color={rgb,255:red,59;green,62;blue,247}, curve={height=-30pt}, from=10-9, to=24-7]
	\arrow["{1 \chi^{-1}}", color={rgb,255:red,59;green,62;blue,247}, from=10-8, to=12-8]
	\arrow["\psi1", color={rgb,255:red,214;green,92;blue,92}, curve={height=-12pt}, from=12-8, to=24-7]
	\arrow["1\chi", color={rgb,255:red,59;green,62;blue,247}, from=7-8, to=10-9]
	\arrow["{1(\chi\chi)}", shift left, color={rgb,255:red,59;green,62;blue,247}, from=5-3, to=7-4]
	\arrow["{1(1\chi)}"{description}, shift left, color={rgb,255:red,59;green,62;blue,247}, from=7-4, to=5-5]
	\arrow["{1(\chi\chi)}"', shift right, color={rgb,255:red,59;green,62;blue,247}, from=5-6, to=5-5]
	\arrow["\psi1"{description}, color={rgb,255:red,214;green,92;blue,92}, from=12-8, to=19-7]
	\arrow["\psi", color={rgb,255:red,214;green,92;blue,92}, from=19-7, to=27-6]
	\arrow["\psi", color={rgb,255:red,214;green,92;blue,92}, from=24-7, to=27-6]
	\arrow["1\chi"', color={rgb,255:red,59;green,62;blue,247}, from=7-1, to=10-1]
	\arrow["{\text{id}\otimes k_*^q(\psi)}"', color={rgb,255:red,214;green,92;blue,92}, from=7-2, to=9-2]
	\arrow["1\chi"', color={rgb,255:red,59;green,62;blue,247}, from=9-2, to=10-1]
	\arrow["{1\chi^{-1}}"', color={rgb,255:red,59;green,62;blue,247}, curve={height=24pt}, from=10-1, to=18-2]
	\arrow["{1\chi^{-1}}"', color={rgb,255:red,59;green,62;blue,247}, curve={height=30pt}, from=9-2, to=14-2]
	\arrow["{1 q_*(\chi)}", color={rgb,255:red,59;green,62;blue,247}, from=14-2, to=18-2]
	\arrow["\psi"', color={rgb,255:red,214;green,92;blue,92}, from=14-2, to=20-3]
	\arrow["\psi"', color={rgb,255:red,214;green,92;blue,92}, from=18-2, to=20-2]
	\arrow["{q_*(1\chi)}", color={rgb,255:red,59;green,62;blue,247}, from=20-3, to=20-2]
	\arrow["{q_*(1\psi^{-1})}"', color={rgb,255:red,214;green,92;blue,92}, from=20-2, to=25-3]
	\arrow["{q_*(1\chi)}", color={rgb,255:red,59;green,62;blue,247}, from=20-3, to=24-4]
	\arrow["{q_*(1\psi^{-1})}"', color={rgb,255:red,214;green,92;blue,92}, from=24-4, to=25-3]
	\arrow["{q_*(1(1\chi))}"', color={rgb,255:red,59;green,62;blue,247}, from=25-3, to=27-4]
	\arrow["{q_*(\alpha^{\mathcal{C}})}"', color={rgb,255:red,92;green,214;blue,92}, from=27-4, to=27-6]
	\arrow["{q_*(\alpha^{\mathcal{C}})}"', color={rgb,255:red,92;green,214;blue,92}, from=25-3, to=25-5]
	\arrow["{q_*(1\chi)}"', color={rgb,255:red,59;green,62;blue,247}, from=25-5, to=27-6]
	\arrow["1\psi", color={rgb,255:red,214;green,92;blue,92}, from=7-4, to=9-4]
	\arrow["1\chi", color={rgb,255:red,59;green,62;blue,247}, from=9-2, to=9-4]
	\arrow["{1q_*(\chi)}"{description}, color={rgb,255:red,59;green,62;blue,247}, from=14-2, to=11-2]
	\arrow["{1\chi^{-1}}", color={rgb,255:red,59;green,62;blue,247}, from=9-4, to=11-2]
	\arrow["{1\psi^{-1}}"{description}, color={rgb,255:red,214;green,92;blue,92}, from=9-4, to=7-5]
	\arrow["{1( 1\chi)}"{description}, color={rgb,255:red,59;green,62;blue,247}, from=7-5, to=5-5]
	\arrow["{q_*(1\chi)}"{description}, color={rgb,255:red,59;green,62;blue,247}, from=20-3, to=17-3]
	\arrow["\psi", color={rgb,255:red,214;green,92;blue,92}, from=11-2, to=17-3]
	\arrow["{q_*(1\psi^{-1})}"{description}, color={rgb,255:red,214;green,92;blue,92}, from=24-4, to=20-4]
	\arrow["{q_*(1\psi^{-1})}"{description}, color={rgb,255:red,214;green,92;blue,92}, from=17-3, to=20-4]
	\arrow["{\alpha^{\mathcal{C}}}", color={rgb,255:red,92;green,214;blue,92}, from=7-5, to=10-6]
	\arrow["{\alpha^{\mathcal{C}}}"', color={rgb,255:red,92;green,214;blue,92}, from=5-5, to=7-6]
	\arrow["1\chi"{description}, color={rgb,255:red,59;green,62;blue,247}, from=10-6, to=7-6]
	\arrow["{(1\chi)1}"{description}, color={rgb,255:red,59;green,62;blue,247}, from=10-8, to=7-6]
	\arrow["{q_*(\alpha^{\mathcal{C}})}"', color={rgb,255:red,92;green,214;blue,92}, from=20-4, to=25-5]
	\arrow["{q_*(1(1\chi))}"{description}, color={rgb,255:red,59;green,62;blue,247}, from=20-4, to=22-5]
	\arrow["{q_*(\alpha^{\mathcal{C}})}"{description}, color={rgb,255:red,92;green,214;blue,92}, from=22-5, to=27-6]
	\arrow["{(1\chi^{-1})1}"{description}, color={rgb,255:red,59;green,62;blue,247}, from=7-6, to=12-7]
	\arrow["{1\chi^{-1}}", color={rgb,255:red,59;green,62;blue,247}, from=12-7, to=12-8]
	\arrow["\psi1", color={rgb,255:red,214;green,92;blue,92}, from=12-7, to=15-7]
	\arrow["{1\chi^{-1}}"', color={rgb,255:red,59;green,62;blue,247}, from=15-7, to=19-7]
	\arrow["{1\psi^{-1}}", color={rgb,255:red,214;green,92;blue,92}, from=11-2, to=12-3]
	\arrow["\psi", color={rgb,255:red,214;green,92;blue,92}, from=12-3, to=20-4]
	\arrow["{1\psi^{-1}}", color={rgb,255:red,214;green,92;blue,92}, from=12-3, to=10-5]
	\arrow["{1(\chi^{-1}\chi^{-1})}"{description}, color={rgb,255:red,59;green,62;blue,247}, from=7-5, to=10-5]
	\arrow["{(1\chi^{-1})1}", color={rgb,255:red,59;green,62;blue,247}, from=10-6, to=14-6]
	\arrow["1\chi"', color={rgb,255:red,59;green,62;blue,247}, from=14-6, to=12-7]
	\arrow["\psi1"', color={rgb,255:red,214;green,92;blue,92}, from=14-6, to=16-6]
	\arrow["1\chi"', color={rgb,255:red,59;green,62;blue,247}, from=16-6, to=15-7]
	\arrow["{1\chi^{-1}}"', color={rgb,255:red,59;green,62;blue,247}, from=16-6, to=20-6]
	\arrow["{\text{id}\otimes q_*(\chi)}"', color={rgb,255:red,59;green,62;blue,247}, from=20-6, to=19-7]
	\arrow["{\psi^{-1}}", color={rgb,255:red,214;green,92;blue,92}, from=19-5, to=20-6]
	\arrow["{q_*(\alpha^{\mathcal{C}})}", color={rgb,255:red,92;green,214;blue,92}, from=20-4, to=19-5]
	\arrow["{q_*(1\chi)}", color={rgb,255:red,59;green,62;blue,247}, curve={height=-12pt}, from=19-5, to=27-6]
	\arrow["{\alpha^{\mathcal{C}}}"', color={rgb,255:red,92;green,214;blue,92}, from=10-5, to=14-5]
	\arrow["1\chi"', color={rgb,255:red,59;green,62;blue,247}, from=14-5, to=14-6]
	\arrow["\psi1"', color={rgb,255:red,214;green,92;blue,92}, from=14-5, to=17-5]
	\arrow["\psi"', color={rgb,255:red,214;green,92;blue,92}, from=17-5, to=19-5]
	\arrow["1\chi", color={rgb,255:red,59;green,62;blue,247}, from=17-5, to=16-6]
	\arrow["{1\psi^{-1}}", color={rgb,255:red,214;green,92;blue,92}, from=12-3, to=13-4]
	\arrow["{\alpha^{\mathcal{C}}}", color={rgb,255:red,92;green,214;blue,92}, from=13-4, to=14-5]
	\arrow["{(1\chi)\chi}"', color={rgb,255:red,59;green,62;blue,247}, from=7-7, to=7-6]
	\arrow["{1 (\chi\chi)}"{description}, color={rgb,255:red,59;green,62;blue,247}, from=1-5, to=5-5]
	\arrow["{1 (\chi\chi)}", color={rgb,255:red,59;green,62;blue,247}, from=3-4, to=5-5]
	\arrow["{(1\chi^{-1})\chi^{-1}}"{description}, color={rgb,255:red,59;green,62;blue,247}, from=10-6, to=14-5]
\end{tikzcd}\]
\restoregeometry
The commutativity of the ``squares'' are obvious enough, coming from functorality, naturality and our assumption that $\B$ is strictly unital. We do note that the commutativity of the pentagons (labeled \fbox{P.I} and \fbox{P.II} above) 
\[\begin{tikzcd}
	{\q_* \widehat{r}_*V\otimes \q_* \widehat{r}_* k_* W} & {\q_* (\widehat{r}_*V\otimes \widehat{r}_* k_* W)} \\
	& {\q_* \widehat{r}_*(V\otimes  k_* W)} \\
	{(\q\widehat{r})_*V\otimes (\q\widehat{r})_* k_* W} & {(\q\widehat{r})_*(V\otimes  k_* W)}
	\arrow["\psi", color={rgb,255:red,214;green,92;blue,92}, from=1-1, to=1-2]
	\arrow["{\q_*(\psi)}"', color={rgb,255:red,214;green,92;blue,92}, from=1-2, to=2-2]
	\arrow["\chi"', color={rgb,255:red,96;green,96;blue,215}, from=2-2, to=3-2]
	\arrow["\chi\otimes\chi", color={rgb,255:red,96;green,96;blue,215}, from=1-1, to=3-1]
	\arrow["\psi", color={rgb,255:red,214;green,92;blue,92}, from=3-1, to=3-2]
\end{tikzcd}\]
is exactly the requirement that $\chi_{\q,\ar}:\q_* \ar_*\rightarrow (\q\ar)_*$ is a natural isomorphism of monoidal functors\footnote{Recall that $\chi_{\q,\ar}$ is part of the structure of the $G$-action on $\B$.}. The commutativity of the ``hexagon'' (labelled \fbox{H} above)
\[\begin{tikzcd}
	{\q_* V\otimes  (\q_* k_* W\otimes \q_* k_* l_* Z)} & {(\q_* V\otimes  \q_* k_* W)\otimes \q_* k_* l_* Z} \\
	{\q_* V\otimes \q_* (k_* W\otimes k_* l_* Z)} & {\q_*(V\otimes   k_* W)\otimes \q_* k_* l_* Z} \\
	{\q_* (V\otimes  (k_* W\otimes k_* l_* Z))} & {\q_*(( V\otimes  k_* W)\otimes k_*l_* Z)}
	\arrow["\psi", color={rgb,255:red,214;green,92;blue,92}, from=2-1, to=3-1]
	\arrow["{\q_*(\alpha^{\mathcal{C}})}", color={rgb,255:red,92;green,214;blue,92}, from=3-1, to=3-2]
	\arrow["{\psi\otimes\text{id}}"', color={rgb,255:red,214;green,92;blue,92}, from=1-2, to=2-2]
	\arrow["\psi"', color={rgb,255:red,214;green,92;blue,92}, from=2-2, to=3-2]
	\arrow["{\text{id}\otimes\psi}"', color={rgb,255:red,214;green,92;blue,92}, from=1-1, to=2-1]
	\arrow["{\alpha^{\mathcal{C}}}", color={rgb,255:red,92;green,214;blue,92}, from=1-1, to=1-2]
\end{tikzcd}\]
is exactly the requirement that $\psi^{\q}$ is monoidal, which follows from the definition of $G$-action on $\B$. Finally, in order for (\ref{monoidal comm diagram for q}) to commute we need the $\kt$-coefficients on both sides to be equal. To achieve this, consider $\vek(K,\omega)$ as a $\underline{G}$-module with the action defined by
\[g_* k:=g \otimes (k\otimes g\inv).\]
This action clearly defines an autoequivalence of $\vek(K,\omega)$ for each $g\in G$, but we check if it is indeed monoidal. We define natural isomorphisms
\[\psi(g)_{k,l}:g_* k\otimes g_* l\rightarrow g_*(k\otimes l)\]
which gives the diagram 
\[\begin{tikzcd}
	{(g\otimes (k\otimes g^{-1}))\otimes(g\otimes(l\otimes g^{-1}))} & {g\otimes \big((k\otimes g^{-1})\otimes(g\otimes(l\otimes g^{-1}))\big)} \\
	{\big((g\otimes (k\otimes g^{-1}))\otimes g\big)\otimes(l\otimes g^{-1})} & {g\otimes \big(((k\otimes g^{-1})\otimes g)\otimes(l\otimes g^{-1})\big)} \\
	{\big(g\otimes ((k\otimes g^{-1})\otimes g)\big)\otimes(l\otimes g^{-1})} & {g\otimes \big(k\otimes(l\otimes g^{-1})\big)} \\
	{\big(g\otimes k\big)\otimes(l\otimes g^{-1})} \\
	{g\otimes \big(k\otimes(l\otimes g^{-1})\big)} & {g\otimes (kl\otimes g^{-1})}
	\arrow["{\omega(g,kg^{-1},glg^{-1})^{-1}}", curve={height=-18pt}, from=1-1, to=1-2]
	\arrow["{\omega(gkg^{-1},g,lg^{-1})}"', from=1-1, to=2-1]
	\arrow["{1\otimes\omega(kg^{-1},g,lg^{-1})}", from=1-2, to=2-2]
	\arrow["{\omega(g,kg^{-1},g)^{-1}\otimes 1}"', from=2-1, to=3-1]
	\arrow["{1\otimes\omega(k,g^{-1},g)^{-1}}", from=2-2, to=3-2]
	\arrow["{1\otimes \omega(k,g^{-1},g)^{-1}\otimes 1}"', from=3-1, to=4-1]
	\arrow["{1\otimes\omega(k,l,g^{-1})}", from=3-2, to=5-2]
	\arrow["{\omega(g,k,lg^{-1})^{-1}}"', from=4-1, to=5-1]
	\arrow["{1\otimes\omega(k,l,g^{-1})}"', from=5-1, to=5-2]
\end{tikzcd}\]
Because the diagram is comprised of associators, by Mac Lane's coherence theorem \cite{Maclane1971-MACCFT} the diagram must commute. We chose\footnote{We note that while technically we are \textit{choosing} the right side of the diagram to represent the scalar for $\chit$, there is no real choice to be made, as the diagram commutes and thus the coefficients on each side are equal.} $\psit$ to have scalar coming from the right side of the above diagram. Now, considering that the coefficients of each morphism in (\ref{monoidal comm diagram for q}) are defined as the corresponding coefficients of structure morphisms from the $\underline{G}$-action on $\vek(K,\omega)$, we see that, because each coefficient may be viewed as an associator in $\vek(K,\omega)$, the coefficients on either side of the diagram must agree again by Mac Lane.

    For any object $[V,k]\in\BK$ we also have $[q\inv_* V,q\inv kq]\in \BK$ ($q\inv_* V\in \B$ because of the $G$-action on $\B$ and $q\inv kq\in K$ because $K$ is a normal subgroup of $G$) such that 
    \[q_*[q\inv_* V,q\inv kq]=[q_*q\inv_* V,q(q\inv kq)q\inv]\cong [e_* V,k]=[V,k].\]
    Thus $q_*$ is essentially surjective. Finally, $q_*$ is fully faithful if, for each pair of objects $[V,k],[W,l]\in \BK$, the function
    \[q_*:\BK([V,k],[W,l])\rightarrow\BK(q_*[V,k],q_*[W,l])\]
    between hom-sets is bijective. We note that, by definition of $\BK$, 
    \[\BK([V,k],[W,l])=\emptyset \quad \text{if }k\neq l.\]
    Therefore it suffices to consider pairs of objects $[V,k],[W,k]\in \BK$. Again by definition of $\BK$, we have
    \[\BK([V,k],[W,k])=\B(V,W).\]
    Likewise
    \[\BK([q_* V,k^{q}],[q_* W,k^{q}])=\B(q_* V, q_* W).\]
    By definition of $G$-action over $\B$, $q_*$ defines an automorphism of $\B$ hence is fully faithful as a functor
    \[q_*:\B\rightarrow\B\]
    thus there exists a bijection from $\B(V,W)$ to $\B(q_* V,q_* W)$. Therefore $q_*$ is fully faithful as a functor from $\BK$ to itself and $(q_*,\psitq)$ is a monoidal autoequivalence of $\BK$.
\end{proof}

The assignment $q\mapsto q_*$ defines a functor between the underlying monoidal categories of $3\Gr(Q,\pi)$ and $\Aut(\BK)$. Now that we have built monoidal autoequivalences $(q_*,\psitq)$, we want pseudonatural isomorphisms 
\[(U_{q,r},\chi_{q,r}):q_*\circ r_*\Rightarrow (qr)_*\] 
for all $q,r\in Q$. Note that $\q\ar=\gamma(q,r)\qr$ for a uniquely determined (by the lift) function $$\gamma:Q\times Q\rightarrow K.$$
\begin{definition}
    For any $q,r\in Q$ we define maps
    \[\chit_{q,r}[V,k]:=\lambda\big([\chi\inv_{\gamma(q,r),\qr},\gamma(q,r)k^{\qr}]\circ[\chi_{\q,\ar}(V),\gamma(q,r) k^{\qr}]\big): q_* r_*[V,k]\cdot U_{q,r}\rightarrow U_{q,r}\cdot (qr)_*[V,k
    ]\]
    where  $$\lambda=\frac{\omega(\q,\ar,k\ar\inv\q\inv)\omega(k,\qr\inv,\gamma\inv)\omega(\qr,k\qr\inv,\gamma\inv)}{\omega(\ar,k\ar\inv,\q\inv)\omega(k,\ar\inv,\q\inv)\omega(\gamma,\qr,k\qr\inv \gamma\inv)\omega(\gamma,k^{\qr}\gamma\inv,\gamma)\omega(k^{\qr},\gamma\inv,\gamma)}$$
    referring to $\gamma(q,r)$ as $\gamma$ for brevity. We also define $U_{q,r}:=[I,\gamma(q,r)]$, for all $q,r\in Q$.
\end{definition}

\begin{lemma}
    For any $q,r\in Q$, the maps $(U_{q,r},\chit_{q,r})$ have the structure of a linear pseudonatural isomorphism of monoidal autoequivalences of $\BK$ 
    \begin{align*}
        \chit_{q,r}:q_* r_*&\Rightarrow (qr)_*\\
        q_* r_*[V,k]\cdot U_{q,r} &\mapsto U_{q,r}\cdot (qr)_*[V,k].
    \end{align*}
    
\end{lemma}
\bp Pseudonaturality of $\chit$ requires the following diagram to commute
\[\begin{tikzcd}
	{q_* r_*[V,k]\cdot U_{q,r}} & {U_{q,r}\cdot (qr)_*[V,k]} \\
	{q_* r_*[W,k]\cdot U_{q,r}} & {U_{q,r}\cdot (qr)_*[W,k]}
	\arrow["{\tilde{\chi}}", from=1-1, to=1-2]
	\arrow["{q_* r_*[f,k]\otimes 1}"', from=1-1, to=2-1]
	\arrow["{1\otimes(qr)_*[f,k]}", from=1-2, to=2-2]
	\arrow["{\tilde{\chi}}"', from=2-1, to=2-2]
\end{tikzcd}\]
This follows from the fact that $\chi$, from the $G$-action on $\B$, is natural. In order for $\chit_{q,r}$ to be a pseudonatural isomorphism we need the following diagram to commute

\begin{equation}\label{axiom for chi}  
\begin{tikzcd}
	{(q_* r_*[V,k]\cdot q_* r_*[W,l])\cdot U_{q,r}} && {q_* (r_*[V,k]\cdot r_*[W,l])\cdot U_{q,r}} \\
	{q_* r_*[V,k]\cdot (q_* r_*[W,l]\cdot U_{q,r})} \\
	{q_* r_*[V,k]\cdot(U_{q,r}\cdot  (qr)_*[W,l])} && {q_* r_*([V,k]\cdot [W,l])\cdot U_{q,r}} \\
	{(q_* r_*[V,k]\cdot U_{q,r})\cdot  (qr)_*[W,l]} \\
	{(U_{q,r}\cdot(qr)_*[V,k])\cdot (qr)_*[W,l]} \\
	{U_{q,r}\cdot((qr)_*[V,k]\cdot (qr)_*[W,l])} && {U_{q,r}\cdot(qr)_*([V,k]\cdot[W,l])}
	\arrow["{\psit\otimes\text{id}}"', from=1-1, to=1-3]
	\arrow["{(\alpha^{\omega})\inv}", from=1-1, to=2-1]
	\arrow["{q_*(\psit)\otimes\text{id}}"', from=1-3, to=3-3]
	\arrow["{\text{id}\otimes \chit}", from=2-1, to=3-1]
	\arrow["{\alpha^{\omega}}", from=3-1, to=4-1]
	\arrow["\chit"', from=3-3, to=6-3]
	\arrow["{\chit\otimes\text{id}}", from=4-1, to=5-1]
	\arrow["{(\alpha^{\omega})^{-1}}", from=5-1, to=6-1]
	\arrow["{\text{id}\otimes\psit}", from=6-1, to=6-3]
\end{tikzcd}\end{equation}
We note that because $\B$ is unital and $U_{q,r}=[I,\gamma(q,r)]$ that by definition the first associator in the above diagram is simply scalar times identity. We note the reduced composites of the other two associators below
\[
    \alpha^{\omega}(q_* r_*[V,k],U_{q,r},(qr)_*[W,l])=[1\otimes \chi_{k^{\q\ar},\gamma(q,r)},\gamma(q,r)(kl)^{\qr}]\]
\[
    \alpha^{\omega}(U_{q,r},(qr)_*[V,k],(qr)_*[W,l])\inv=[(\psi^{\gamma(q,r)}_{\qr_* V,\qr_* W})\inv,\gamma(q,r)(kl)^{\qr}]\inv \circ [1\otimes \chi_{\gamma(q,r),k^{\qr}},\gamma(q,r)(kl)^{\qr}]\inv.
\]
Computing the first term of diagram (\ref{axiom for chi}) gives 
\[(q_* r_*[V,k]\cdot q_* r_*[W,l])\cdot U_{q,r}=[\q_* \ar_*V\otimes k^{\q\ar}_*\q_* \ar_* W,\gamma(q,r)(kl)^{\qr}].\]
The $K$-component stays unchanged throughout the diagram, thus we only must consider the $\B$-component. Expanding our composites gives the diagram 
\[
\begin{tikzcd}
	{\widehat{q}_* \widehat{r}_*V\otimes (k^{\widehat{q}\widehat{r}})_*\widehat{q}_* \widehat{r}_*W} & {\widehat{q}_* \widehat{r}_*V\otimes (\widehat{q}k^{\widehat{r}})_* \widehat{r}_*W} \\
	{\widehat{q}_* \widehat{r}_*V\otimes (k^{\widehat{q}\widehat{r}})_*(\widehat{q}\widehat{r})_*W} & {\widehat{q}_* \widehat{r}_*V\otimes \widehat{q}_*(k^{\widehat{r}})_* \widehat{r}_*W} \\
	{\widehat{q}_* \widehat{r}_*V\otimes (k^{\widehat{q}\widehat{r}})_*\gamma(q,r)_*(\widehat{qr})_*W} & {\widehat{q}_* (\widehat{r}_*V\otimes (k^{\widehat{r}})_* \widehat{r}_*W)} \\
	{\widehat{q}_* \widehat{r}_*V\otimes (\gamma(q,r)k^{\widehat{qr}})_*(\widehat{qr})_*W} & {\widehat{q}_* (\widehat{r}_*V\otimes (\widehat{r}k)_* W)} \\
	{(\widehat{q}\widehat{r})_*V\otimes (\gamma(q,r)k^{\widehat{qr}})_*(\widehat{qr})_*W} & {\widehat{q}_* (\widehat{r}_*V\otimes \widehat{r}_* k_* W)} \\
	{\gamma(q,r)_*\widehat{qr}_*V\otimes (\gamma(q,r)k^{\widehat{qr}})_*(\widehat{qr})_*W} & {\widehat{q}_* \widehat{r}_*(V\otimes  k_* W)} \\
	{\gamma(q,r)_*\widehat{qr}_*V\otimes \gamma(q,r)_* k^{\widehat{qr}}_*(\widehat{qr})_*W} & {(\widehat{q}\widehat{r})_*(V\otimes  k_* W)} \\
	{\gamma(q,r)_*(\widehat{qr}_*V\otimes k^{\widehat{qr}}_*(\widehat{qr})_*W)} \\
	{\gamma(q,r)_*(\widehat{qr}_*V\otimes (\widehat{qr}k)_* W)} \\
	{\gamma(q,r)_*(\widehat{qr}_*V\otimes \widehat{qr}_* k_* W)} & {\gamma(q,r)_*\widehat{qr}_*(V\otimes  k_* W)}
	\arrow["{1\otimes \chi}", from=1-1, to=1-2]
	\arrow["{1\otimes (k^{\widehat{q}\widehat{r}})_*(\chi)}"', from=1-1, to=2-1]
	\arrow["{1\otimes \chi^{-1}}", from=1-2, to=2-2]
	\arrow["{1\otimes (k^{\widehat{q}\widehat{r}})_*(\chi^{-1})}"', from=2-1, to=3-1]
	\arrow["\psi", from=2-2, to=3-2]
	\arrow["{1\otimes\chi}"', from=3-1, to=4-1]
	\arrow["{q_*(1\otimes\chi)}", from=3-2, to=4-2]
	\arrow["{\chi\otimes 1}"', from=4-1, to=5-1]
	\arrow["{q_*(1\otimes\chi^{-1})}", from=4-2, to=5-2]
	\arrow["{\chi^{-1}\otimes 1}"', from=5-1, to=6-1]
	\arrow["{q_*(\psi)}", from=5-2, to=6-2]
	\arrow["{1\otimes \chi^{-1}}"', from=6-1, to=7-1]
	\arrow["\chi", from=6-2, to=7-2]
	\arrow["\psi"', from=7-1, to=8-1]
	\arrow["{\chi^{-1}}", from=7-2, to=10-2]
	\arrow["{\gamma(q,r)_*(1\otimes\chi)}"', from=8-1, to=9-1]
	\arrow["{\gamma(q,r)_*(1\otimes\chi^{-1})}"', from=9-1, to=10-1]
	\arrow["{\gamma(q,r)_*(\psi)}"', from=10-1, to=10-2]
\end{tikzcd}\]
Below we make use of the commutativity of the following diagram in $\B$
\[\begin{tikzcd}
	{q_*r_* V\otimes q_* r_* W} & {(qr)_* V\otimes q_* r_* W} & {(qr)_* V\otimes (qr)_* W}
	\arrow["{\chi\otimes \text{id}}"', from=1-1, to=1-2]
	\arrow["{\text{id}\otimes \chi}"', from=1-2, to=1-3]
	\arrow["\chi\otimes\chi", curve={height=-24pt}, from=1-1, to=1-3]
\end{tikzcd}\]
Using the same color-coding and space-saving measures as before, we see that the above diagram commutes

\[\begin{tikzcd}[cramped,scale cd=0.5,column sep = 0.25em, center picture]
	&& {\widehat{q}_*\widehat{r}_* V\otimes k^{\widehat{q}\widehat{r}}_*\widehat{q}_*\widehat{r}_* W} & {\widehat{q}_*\widehat{r}_* V\otimes (\widehat{q}\widehat{r}k\widehat{r}^{-1})_*\widehat{r}_* W} \\
	&&&& {\widehat{q}_*\widehat{r}_* V\otimes \widehat{q}_*(\widehat{r}k\widehat{r}^{-1})_*\widehat{r}_* W} & {\widehat{q}_*(\widehat{r}_* V\otimes (\widehat{r}k\widehat{r}^{-1})_*\widehat{r}_* W)} \\
	& {(\widehat{q}\widehat{r})_* V\otimes k^{qr}_*(\widehat{q}\widehat{r})_* W} \\
	\\
	{\gamma(q,r)_*\widehat{qr}_* V\otimes k^{qr}_*\gamma(q,r)_*\widehat{qr}_* W} &&& {(\widehat{q}\widehat{r})_* V\otimes (\widehat{q}\widehat{r}k)_* W} &&& {\widehat{q}_*(\widehat{r}_* V\otimes (\widehat{r}k)_* W)} \\
	&&&& {(\widehat{q}\widehat{r})_* V\otimes \widehat{q}_*(\widehat{r}k)_* W} & {\widehat{q}_*\widehat{r}_* V\otimes \widehat{q}_*(\widehat{r}k)_* W} \\
	&&& {(\widehat{q}\widehat{r})_* V\otimes (\widehat{q}\widehat{r})_*k_* W} \\
	{\gamma(q,r)_*\widehat{qr}_* V\otimes (\gamma(q,r)k^{\widehat{qr}})_*\widehat{qr}_* W} && {\gamma(q,r)_*\widehat{qr}_* V\otimes (\widehat{q}\widehat{r})_*(k\widehat{qr}^{-1})_*\widehat{qr}_* W} \\
	&&&& {(\widehat{q}\widehat{r})_* V\otimes \widehat{q}_*\widehat{r}_*k_* W} & {\widehat{q}_*\widehat{r}_* V\otimes \widehat{q}_* \widehat{r}_* k_* W} && {\widehat{q}_*(\widehat{r}_* V\otimes \widehat{r}_* k_* W)} \\
	\\
	& {\gamma(q,r)_*\widehat{qr}_* V\otimes (\gamma(q,r)\widehat{qr}k)_* W} & {\gamma(q,r)_*\widehat{qr}_* V\otimes (\widehat{q}\widehat{r})_* k_* W} &&&& {\fbox{P.I}} \\
	&&& {(\widehat{q}\widehat{r})_* V\otimes (\widehat{q}\widehat{r}k)_* W} && {(\widehat{q}\widehat{r})_* V\otimes (\widehat{q}\widehat{r})_* k_* W} && {\widehat{q}_*\widehat{r}_* (V\otimes  k_* W)} \\
	&& {\gamma(q,r)_*\widehat{qr}_* V\otimes \gamma(q,r)_*\widehat{qr}_* k_* W} && {\gamma(q,r)_*\widehat{qr}_* V\otimes \gamma(q,r)_*\widehat{qr}_* k_* W} \\
	& {\gamma(q,r)_*\widehat{qr}_* V\otimes \gamma(q,r)_*(\widehat{qr}k)_* W} \\
	\\
	&&& {\gamma(q,r)_*\widehat{qr}_* V\otimes \gamma(q,r)_*(\widehat{qr}k)_* W} && {\fbox{P.II}} && {(\widehat{q}\widehat{r})_* (V\otimes  k_* W)} \\
	& {\gamma(q,r)_*\widehat{qr}_* V\otimes \gamma(q,r)_*(k^{\widehat{qr}})_*\widehat{qr}_* W} \\
	\\
	&& {\gamma(q,r)_*(\widehat{qr}_* V\otimes (k^{\widehat{qr}})_*\widehat{qr}_* W)} &&&& {\gamma(q,r)_*\widehat{qr}_* (V\otimes k_*W)} \\
	&&&& {\gamma(q,r)_*(\widehat{qr}_* V\otimes \widehat{qr}_*k_*W)} \\
	&&& {\gamma(q,r)_*(\widehat{qr}_* V\otimes (\widehat{qr}k)_*W)}
	\arrow["{\chi\otimes k^{qr}_*(\chi)}"{description}, color={rgb,255:red,92;green,92;blue,214}, from=1-3, to=3-2]
	\arrow["{\chi^{-1}\otimes k^{qr}_*(\chi^{-1})}"{description}, color={rgb,255:red,92;green,92;blue,214}, from=3-2, to=5-1]
	\arrow["{1\otimes\chi^{-1}}"', color={rgb,255:red,92;green,92;blue,214}, curve={height=30pt}, from=8-1, to=17-2]
	\arrow["\psi"{description}, color={rgb,255:red,214;green,92;blue,92}, from=17-2, to=19-3]
	\arrow["{\gamma(q,r)_*(1\otimes \chi)}"{description}, color={rgb,255:red,92;green,92;blue,214}, from=19-3, to=21-4]
	\arrow["{\gamma(q,r)_*(1\otimes \chi^{-1})}"{description}, color={rgb,255:red,92;green,92;blue,214}, from=21-4, to=20-5]
	\arrow["{\gamma(q,r)_*(\psi)}"{description}, color={rgb,255:red,214;green,92;blue,92}, from=20-5, to=19-7]
	\arrow["{1\otimes \chi}", color={rgb,255:red,92;green,92;blue,214}, from=1-3, to=1-4]
	\arrow["{1\otimes \chi^{-1}}"{description}, color={rgb,255:red,92;green,92;blue,214}, from=1-4, to=2-5]
	\arrow["\psi", color={rgb,255:red,214;green,92;blue,92}, from=2-5, to=2-6]
	\arrow["{\widehat{q}_*(1\otimes\chi)}"{description}, color={rgb,255:red,92;green,92;blue,214}, from=2-6, to=5-7]
	\arrow["{\widehat{q}_*(1\otimes\chi^{-1})}"{description}, color={rgb,255:red,92;green,92;blue,214}, from=5-7, to=9-8]
	\arrow["{\widehat{q}_*(\psi)}", color={rgb,255:red,214;green,92;blue,92}, from=9-8, to=12-8]
	\arrow["\chi", color={rgb,255:red,92;green,92;blue,214}, from=12-8, to=16-8]
	\arrow["{\chi^{-1}}"{description}, color={rgb,255:red,92;green,92;blue,214}, from=16-8, to=19-7]
	\arrow["{1\otimes \chi}"', color={rgb,255:red,92;green,92;blue,214}, from=5-1, to=8-1]
	\arrow["{\chi\otimes \chi}"', color={rgb,255:red,92;green,92;blue,214}, from=1-4, to=5-4]
	\arrow["1\otimes\chi"{description}, color={rgb,255:red,92;green,92;blue,214}, from=3-2, to=5-4]
	\arrow["{\chi^{-1}\otimes\chi^{-1}}"{description}, color={rgb,255:red,92;green,92;blue,214}, from=5-4, to=8-1]
	\arrow["\chi\otimes\chi"', color={rgb,255:red,92;green,92;blue,214}, from=2-5, to=6-5]
	\arrow["{1\otimes \widehat{q}_*(\chi^{-1})}"{description}, color={rgb,255:red,92;green,92;blue,214}, from=5-4, to=6-5]
	\arrow["{\chi^{-1}\otimes 1}", color={rgb,255:red,92;green,92;blue,214}, from=6-5, to=6-6]
	\arrow["\psi"{description}, color={rgb,255:red,214;green,92;blue,92}, from=6-6, to=5-7]
	\arrow["{1\otimes \chi}"{description}, color={rgb,255:red,92;green,92;blue,214}, from=2-5, to=6-6]
	\arrow["{1\otimes \widehat{q}_*(\chi^{-1})}", color={rgb,255:red,92;green,92;blue,214}, from=6-6, to=9-6]
	\arrow["\psi", color={rgb,255:red,214;green,92;blue,92}, from=9-6, to=9-8]
	\arrow["\chi\otimes\chi", color={rgb,255:red,92;green,92;blue,214}, from=9-6, to=12-6]
	\arrow["\psi"{description}, color={rgb,255:red,214;green,92;blue,92}, from=12-6, to=16-8]
	\arrow["\chi\otimes\chi"{description}, color={rgb,255:red,92;green,92;blue,214}, from=13-5, to=12-6]
	\arrow["\psi", color={rgb,255:red,214;green,92;blue,92}, from=13-5, to=20-5]
	\arrow["{1\otimes\chi^{-1}}"', color={rgb,255:red,92;green,92;blue,214}, from=6-5, to=9-5]
	\arrow["{\chi^{-1}\otimes 1}", color={rgb,255:red,92;green,92;blue,214}, from=9-5, to=9-6]
	\arrow["{1\otimes \chi}"{description}, color={rgb,255:red,92;green,92;blue,214}, from=9-5, to=12-6]
	\arrow["{1\otimes\chi^{-1}}", color={rgb,255:red,92;green,92;blue,214}, from=5-4, to=7-4]
	\arrow["{1\otimes\chi^{-1}}"{description}, color={rgb,255:red,92;green,92;blue,214}, from=7-4, to=9-5]
	\arrow["{\chi^{-1}\otimes\chi^{-1}}"{description}, color={rgb,255:red,92;green,92;blue,214}, from=7-4, to=13-5]
	\arrow["{1\otimes \gamma(q,r)_*(\chi)}"{description}, color={rgb,255:red,92;green,92;blue,214}, from=17-2, to=16-4]
	\arrow["\psi"', color={rgb,255:red,214;green,92;blue,92}, from=16-4, to=21-4]
	\arrow["{1\otimes \gamma(q,r)_*(\chi^{-1})}"{description}, color={rgb,255:red,92;green,92;blue,214}, from=16-4, to=13-5]
	\arrow["{1\otimes\chi^{-1}}"{description}, color={rgb,255:red,92;green,92;blue,214}, from=8-1, to=8-3]
	\arrow["{\chi^{-1}\otimes (\widehat{q}\widehat{r})_*(\chi^{-1})}"{description}, color={rgb,255:red,92;green,92;blue,214}, from=7-4, to=8-3]
	\arrow["{1\otimes\chi^{-1}}", color={rgb,255:red,92;green,92;blue,214}, from=12-4, to=7-4]
	\arrow["{\chi^{-1}\otimes\chi^{-1}}", color={rgb,255:red,92;green,92;blue,214}, from=12-4, to=16-4]
	\arrow["{1\otimes (\widehat{q}\widehat{r})_*(\chi)}"', color={rgb,255:red,92;green,92;blue,214}, from=8-3, to=11-3]
	\arrow["{\chi^{-1}\otimes\chi^{-1}}"{description}, color={rgb,255:red,92;green,92;blue,214}, curve={height=-12pt}, from=7-4, to=13-3]
	\arrow["{1\otimes\gamma(q,r)_*(\chi^{-1})}"{description}, color={rgb,255:red,92;green,92;blue,214}, from=16-4, to=13-3]
	\arrow["1\otimes\chi", color={rgb,255:red,92;green,92;blue,214}, from=13-3, to=11-3]
	\arrow["{1\otimes\gamma(q,r)_*(\chi^{-1})}"{description}, color={rgb,255:red,92;green,92;blue,214}, from=14-2, to=13-3]
	\arrow["{1\otimes\gamma(q,r)_*(\chi)}"', color={rgb,255:red,92;green,92;blue,214}, from=17-2, to=14-2]
	\arrow["1\otimes\chi"{description}, color={rgb,255:red,92;green,92;blue,214}, from=8-1, to=11-2]
	\arrow["{1\otimes\chi^{-1}}", color={rgb,255:red,92;green,92;blue,214}, from=11-2, to=11-3]
	\arrow["{1\otimes\chi^{-1}}"', color={rgb,255:red,92;green,92;blue,214}, from=11-2, to=14-2]
\end{tikzcd}\]
\newpage

The commutativity of each square and triangle is obvious to see and the pentagons, \fbox{P.I} and \fbox{P.I},  follow from the requirement that $\chi_{q,r}$ are isomorphisms of monoidal functors, like before.

Lastly, we need the $\kt$-coefficients from the left side of the diagram to equal the coefficients from the right side. We go back to considering the $\underline{G}$-action on $\vek(K,\omega)$. We now consider the pseudonatural isomorphisms $(\chi_{\q,\ar},\gamma(q,r)):\q_*\ar_*\rightarrow \qr_*$ where 
\[\chi_{\q,\ar}(k):\q_*\ar_* k \otimes \gamma(q,r)\Rightarrow \gamma(q,r)\otimes \qr_* k.\] For any $k$, this gives the diagram
\[\begin{tikzcd}[scale cd=0.95, center picture]
	{\Big(\widehat{q}\otimes\Big(\big(\widehat{r}\otimes(k\otimes \widehat{r}^{-1})\big)\otimes \widehat{q}^{-1}\Big)\Big)\otimes\gamma(q,r)} & {\Big(\widehat{q}\otimes\Big(\widehat{r}\otimes\big((k\otimes \widehat{r}^{-1})\otimes \widehat{q}^{-1}\big)\Big)\Big)\otimes \gamma(q,r)} \\
	{\Big(\Big(\widehat{q}\otimes\big(\widehat{r}\otimes(k\otimes \widehat{r}^{-1})\big)\Big)\otimes \widehat{q}^{-1}\Big)\otimes\gamma(q,r)} & {\Big(\widehat{q}\otimes\Big(\widehat{r}\otimes\big(k\otimes (\widehat{qr}^{-1}\otimes \gamma(q,r)^{-1})\big)\Big)\Big)\otimes\gamma(q,r)} \\
	{\Big(\Big((\gamma(q,r)\otimes\widehat{qr})\otimes(k\otimes \widehat{r}^{-1})\Big)\otimes \widehat{q}^{-1}\Big)\otimes\gamma(q,r)} & {\Big(\big(\gamma(q,r)\otimes\widehat{qr}\big)\otimes\big(k\otimes (\widehat{qr}^{-1}\otimes \gamma(q,r)^{-1})\big)\Big)\otimes\gamma(q,r)} \\
	{\Big((\gamma(q,r)\otimes\widehat{qr})\otimes\Big((k\otimes \widehat{r}^{-1})\otimes \widehat{q}^{-1}\Big)\Big)\otimes\gamma(q,r)} & {\big(\gamma(q,r)\otimes\widehat{qr}\big)\otimes\Big(\big(k\otimes (\widehat{qr}^{-1}\otimes \gamma(q,r)^{-1})\big)\otimes\gamma(q,r)\Big)} \\
	{\Big((\gamma(q,r)\otimes\widehat{qr})\otimes\Big(k\otimes (\widehat{qr}^{-1}\otimes \gamma(q,r)^{-1})\Big)\Big)\otimes\gamma(q,r)} & {\big(\gamma(q,r)\otimes\widehat{qr}\big)\otimes\Big(k\otimes \big((\widehat{qr}^{-1}\otimes \gamma(q,r)^{-1})\otimes\gamma(q,r)\big)\Big)} \\
	{\Big((\gamma(q,r)\otimes\widehat{qr})\otimes\Big((k\otimes \widehat{qr}^{-1})\otimes \gamma(q,r)^{-1}\Big)\Big)\otimes\gamma(q,r)} & {\big(\gamma(q,r)\otimes\widehat{qr}\big)\otimes\big(k\otimes \widehat{qr}^{-1}\big)} \\
	{(\gamma(q,r)\otimes\widehat{qr})\otimes\Big(\Big((k\otimes \widehat{qr}^{-1})\otimes \gamma(q,r)^{-1}\Big)\otimes\gamma(q,r)\Big)} \\
	{(\gamma(q,r)\otimes\widehat{qr})\otimes(k\otimes \widehat{qr}^{-1})} & {\gamma(q,r)\otimes\big(\widehat{qr}\otimes(k\otimes \widehat{qr}^{-1})\big)}
	\arrow["{1\otimes\omega(\widehat{r},k\widehat{r}^{-1},\widehat{q}^{-1})^{-1}\otimes 1}", curve={height=-18pt}, from=1-1, to=1-2]
	\arrow["{\omega(\widehat{q},k^{\widehat{r}},\widehat{q}^{-1})\otimes 1}"', from=1-1, to=2-1]
	\arrow["{1\otimes 1\otimes \omega(k,\widehat{r}^{-1},\widehat{q}^{-1})^{-1}\otimes 1}", from=1-2, to=2-2]
	\arrow["{\omega(\widehat{q},\widehat{r},k\widehat{r}^{-1})\otimes 1}"', from=2-1, to=3-1]
	\arrow["{\omega(\widehat{q},\widehat{r},k\widehat{qr}^{-1}\gamma(q,r)^{-1})\otimes 1}", from=2-2, to=3-2]
	\arrow["{\omega(\gamma(q,r)\widehat{qr},k\widehat{r}^{-1},\widehat{q}^{-1})^{-1}\otimes1}"', from=3-1, to=4-1]
	\arrow["{\omega(\gamma(q,r)\widehat{qr},k\widehat{qr}^{-1}\gamma(q,r)^{-1},\gamma(q,r))^{-1}}", from=3-2, to=4-2]
	\arrow["{1\otimes\omega(k,\widehat{r}^{-1},\widehat{q}^{-1})^{-1}\otimes 1}"', shift right, from=4-1, to=5-1]
	\arrow["{1\otimes\omega(k,\widehat{qr}^{-1}\gamma(q,r)^{-1},\gamma(q,r))^{-1}}", from=4-2, to=5-2]
	\arrow["{1\otimes\omega(k,\widehat{qr}^{-1},\gamma(q,r)^{-1})\otimes 1}"', from=5-1, to=6-1]
	\arrow["{1\otimes 1\otimes\omega(\widehat{qr}^{-1},\gamma(q,r)^{-1},\gamma(q,r))}", from=5-2, to=6-2]
	\arrow["{\omega(\gamma(q,r)\widehat{qr},k\widehat{qr}^{-1}\gamma(q,r)^{-1},\gamma(q,r))^{-1}}"', from=6-1, to=7-1]
	\arrow["{\omega(\gamma(q,r),\widehat{qr},k\widehat{qr}^{-1})}", from=6-2, to=8-2]
	\arrow["{1\otimes\omega(k\widehat{qr}^{-1},\gamma(q,r)^{-1},\gamma(q,r))^{-1}}"', from=7-1, to=8-1]
	\arrow["{\omega(\gamma(q,r),\widehat{qr},k\widehat{qr}^{-1})}"', from=8-1, to=8-2]
\end{tikzcd}\]
which commutes by Mac Lane's coherence theorem \cite{Maclane1971-MACCFT}. We have picked the coefficient of each $\chit$ to correspond to the right side of the above diagram.

Now, considering that the coefficients of each morphism in (\ref{axiom for chi}) are defined as the corresponding coefficients of structure morphisms from the $\underline{G}$-action on $\vek(K,\omega)$, we see that, because each coefficient may be viewed as an associator in $\vek(K,\omega)$, the coefficients on either side of the diagram must agree again by Mac Lane.
\epp

\newpage

\begin{definition}
    For any $q,r,s\in Q$, we define maps
\[\Omega_{q,r,s}:=\omega(\q,\ar,\s)\id:(\alpha^{\omega}_{q,r,s})\circ\chit_{q,rs}\circ (\id_{q_*}\otimes \chit_{r,s})\Rrightarrow \chit_{qr,s}\circ (\chit_{q,r}\otimes \id_{s_*})\]\
i.e.
\[\begin{tikzcd}
	{q_*r_*s_*} & {(qr)_*s_*} \\
	{q_*(rs)_*} \\
	{(q(rs))_*} & {((qr)s)_*}
	\arrow["{\chi_{q,r}\otimes \text{id}}", Rightarrow, from=1-1, to=1-2]
	\arrow["{\text{id}\otimes \chi_{r,s}}"', Rightarrow, from=1-1, to=2-1]
	\arrow["{\chi_{qr,s}}", Rightarrow, from=1-2, to=3-2]
	\arrow["{\chi_{q,rs}}"', Rightarrow, from=2-1, to=3-1]
	\arrow["{\Omega_{q,r,s}}"{description}, shorten <=8pt, shorten >=8pt, Rightarrow, scaling nfold=3, from=3-1, to=1-2]
	\arrow["{(\alpha^{\omega}_{q,r,s})_*}"', Rightarrow, from=3-1, to=3-2]
\end{tikzcd}\]
\end{definition}

\begin{lemma}
    For any $q,r,s\in Q$, the maps $\Omega_{q,r,s}$ have the structure of invertible modifications of pseudonatural isomorphisms of monoidal autoequivalences of $\BK$.
\end{lemma}

\bp Let $A=(\alpha^{\omega}_{q,r,s})\circ\chit_{q,rs}\circ (\id_{q_*}\otimes \chit_{r,s})$ and $B=\chit_{qr,s}\circ (\chit_{q,r}\otimes \id_{s_*})$. The composite of pseudonatural isomorphisms is indeed a pseudonatural isomorphism \cite{gurskialgtri}. One may check that the component 1-cells of $A$ and $B$ are $q_* U_{r,s}\cdot U_{q,rs}$ and $U_{q,r}\cdot U_{qr,s}$, respectively. We note
\begin{align*}
    q_* U_{r,s}\cdot U_{q,rs}&=[\q\gamma(r,s)\q\inv]\cdot[\gamma(q,rs)]\\
    &=[\q\ar\s(\widehat{rs})\inv \q\inv]\cdot [\q \widehat{rs}(\widehat{qrs})\inv]\\
    &=U_{q,r}\cdot [\widehat{qr}\s (\widehat{qrs})\inv]\\
    &=U_{q,r}\cdot [\widehat{qr}\s (\gamma(qr,s)\inv \widehat{qr}\s)\inv]\\
    &= U_{q,r}\cdot U_{qr,s}.
\end{align*}
The modification axiom asks 
\[\begin{tikzcd}
	{q_*r_*s_*(\bullet)} && {q_*r_*s_*(\bullet)} && {q_*r_*s_*(\bullet)} && {q_*r_*s_*(\bullet)} \\
	\\
	{((qr)s)_*(\bullet)} && {((qr)s)_*(\bullet)} && {((qr)s)_*(\bullet)} && {((qr)s)_*(\bullet)}
	\arrow[""{name=0, anchor=center, inner sep=0}, "{q_*r_*s_*[V,k]}", from=1-1, to=1-3]
	\arrow[""{name=1, anchor=center, inner sep=0}, "{((qr)s)_*[V,k]}"', from=3-1, to=3-3]
	\arrow["{q_* U_{r,s}\cdot U_{q,rs}}"', curve={height=12pt}, from=1-1, to=3-1]
	\arrow[""{name=2, anchor=center, inner sep=0}, "{q_* U_{r,s}\cdot U_{q,rs}}"', curve={height=18pt}, from=1-3, to=3-3]
	\arrow[""{name=3, anchor=center, inner sep=0}, "{U_{q,r}\cdot U_{qr,s}}"{pos=0.7}, curve={height=-18pt}, from=1-3, to=3-3]
	\arrow[""{name=4, anchor=center, inner sep=0}, "{q_*r_*s_*[V,k]}", from=1-5, to=1-7]
	\arrow[""{name=5, anchor=center, inner sep=0}, "{q_* U_{r,s}\cdot U_{q,rs}}"'{pos=0.3}, curve={height=18pt}, from=1-5, to=3-5]
	\arrow[""{name=6, anchor=center, inner sep=0}, "{U_{q,r}\cdot U_{qr,s}}", curve={height=-18pt}, from=1-5, to=3-5]
	\arrow[""{name=7, anchor=center, inner sep=0}, "{((qr)s)_*[V,k]}"', from=3-5, to=3-7]
	\arrow["{U_{q,r}\cdot U_{qr,s}}", curve={height=-12pt}, from=1-7, to=3-7]
	\arrow["A"'{pos=0.4}, curve={height=18pt}, shorten <=10pt, shorten >=10pt, Rightarrow, from=0, to=1]
	\arrow["\id"', shorten <=7pt, shorten >=7pt, Rightarrow, from=5, to=6]
	\arrow["B"{pos=0.6}, curve={height=-18pt}, shorten <=10pt, shorten >=10pt, Rightarrow, from=4, to=7]
	\arrow["\id", shorten <=7pt, shorten >=7pt, Rightarrow, from=2, to=3]
	\arrow[shorten <=10pt, shorten >=10pt, Rightarrow, no head, from=3, to=5]
\end{tikzcd}\]
for each $[V,k]\in \BK$. The modification axiom is the same as the commutative diagram

\begin{equation}\label{can}
\begin{tikzcd}
	{q_*r_*s_*[V,k]\cdot (q_*U_{r,s}\cdot U_{q,rs})} & {q_*r_*s_*[V,k]\cdot (U_{q,r}\cdot U_{qr,s})} \\
	{(q_*U_{r,s}\cdot U_{q,rs})\cdot((qr)s)_*[V,k]} & {(U_{q,r}\cdot U_{qr,s}) \cdot ((qr)s)_*[V,k]}
	\arrow["{\text{id}}", Rightarrow, from=1-1, to=1-2]
	\arrow["A"', Rightarrow, from=1-1, to=2-1]
	\arrow["B", Rightarrow, from=1-2, to=2-2]
	\arrow["{\text{id}}"', Rightarrow, from=2-1, to=2-2]
\end{tikzcd}
\end{equation}
Clearly $\Omega$ is invertible as $A$ and $B$ are both isomorphisms. To see the above diagram commutes we need to expand $A$ and $B$ into their respective composites:

\begin{equation}\label{Omega axiom}
\begin{tikzcd}
	{q_*r_*s_*[V,k]\cdot (q_*U_{r,s}\cdot U_{q,rs})} & {q_*r_*s_*[V,k]\cdot (U_{q,r}\cdot U_{qr,s})} \\
	{(q_*r_*s_*[V,k]\cdot q_*U_{r,s})\cdot U_{q,rs}} & {(q_*r_*s_*[V,k]\cdot U_{q,r})\cdot U_{qr,s}} \\
	{q_*(r_*s_*[V,k]\cdot U_{r,s})\cdot U_{q,rs}} & {(U_{q,r}\cdot (qr)_*s_*[V,k])\cdot U_{qr,s}} \\
	{q_*(U_{r,s}\cdot (rs)_*[V,k])\cdot U_{q,rs}} & {U_{q,r}\cdot ((qr)_*s_*[V,k]\cdot U_{qr,s})} \\
	{(q_* U_{r,s}\cdot q_*(rs)_*[V,k])\cdot U_{q,rs}} & {U_{q,r}\cdot (U_{qr,s}\cdot (qrs)_*[V,k])} \\
	{q_* U_{r,s}\cdot (q_*(rs)_*[V,k]\cdot U_{q,rs})} & {(U_{q,r}\cdot U_{qr,s}) \cdot (qrs)_*[V,k]} \\
	{q_* U_{r,s}\cdot (U_{q,rs}\cdot (qrs)_*[V,k])} \\
	{(q_* U_{r,s}\cdot U_{q,rs})\cdot (qrs)_*[V,k]} & {(q_*U_{r,s}\cdot U_{q,rs})\cdot(qrs)_*[V,k]}
	\arrow["{\text{id}}", Rightarrow, from=1-1, to=1-2]
	\arrow["{\alpha^{\omega}}"', Rightarrow, from=1-1, to=2-1]
	\arrow["{\alpha^{\omega}}", Rightarrow, from=1-2, to=2-2]
	\arrow["{\tilde{\psi}\otimes \text{id}}"', Rightarrow, from=2-1, to=3-1]
	\arrow["{\tilde{\chi}\otimes\text{id}}", Rightarrow, from=2-2, to=3-2]
	\arrow["{q_*(\tilde{\chi})\otimes\text{id}}"', Rightarrow, from=3-1, to=4-1]
	\arrow["{(\alpha^{\omega})^{-1}}", Rightarrow, from=3-2, to=4-2]
	\arrow["{\tilde{\psi}^{-1}\otimes\text{id}}"', Rightarrow, from=4-1, to=5-1]
	\arrow["{\text{id}\otimes\tilde{\chi}}", Rightarrow, from=4-2, to=5-2]
	\arrow["{(\alpha^{\omega})^{-1}}"', Rightarrow, from=5-1, to=6-1]
	\arrow["{\alpha^{\omega}}", Rightarrow, from=5-2, to=6-2]
	\arrow["{\text{id}\otimes\tilde{\chi}}"', Rightarrow, from=6-1, to=7-1]
	\arrow["{\alpha^{\omega}}"', Rightarrow, from=7-1, to=8-1]
	\arrow["{\text{id}\otimes (\alpha_{q,r,s})_*}"', Rightarrow, from=8-1, to=8-2]
	\arrow["{\text{id}}"', Rightarrow, from=8-2, to=6-2]
\end{tikzcd}
\end{equation}
We have at the level of 2-cells in $\Aut(\BK)$
\begin{align}
    \psitq_{[V,k],[W,l]}&=\left(\frac{\omega(k\q\inv,\q,l\q\inv)\omega(k,l,\q\inv)}{\omega(\q,k\q\inv,l^{\q})\omega(k,\q\inv,\q)}\right)\id,\\
     \psitq_{[V,k],[W,l]}&=\left(\frac{\omega(k\q\inv,\q,l\q\inv)\omega(k,l,\q\inv)}{\omega(\q,k\q\inv,l^{\q})\omega(k,\q\inv,\q)}\right)([\chi_{\q,k}(W)^{-1},klm]\circ[\chi_{k^{\q},\q}(W),klm]),\\
    \alpha^{\omega}_{[V,k],[W,l],[Z,m]}&=\omega(k,l,m)\id,\text{ and }\\
    (\alpha_{q,r,s})_*&=\id 
\end{align}
when 
\begin{enumerate}
    \item[(15)] $W$ is the tensor unit in $\B$,
    \item[(16)] $V$ is the tensor unit in $\B$,
    \item[(17)] any two of $V,W,Z$ are the tensor unit in $\B$, 
    \item[(18)] by definition and functorality.
\end{enumerate}
Expanding composites gives the diagram...
\[\begin{tikzcd}
	{[\widehat{q}_*\widehat{r}_*\widehat{s}_* V,k^{\widehat{q}\widehat{r}\widehat{s}}\gamma(r,s)^{\widehat{q}}\gamma(q,rs)]} & {[\widehat{q}_*\widehat{r}_*\widehat{s}_* V,k^{\widehat{q}\widehat{r}\widehat{s}}\gamma(q,r)\gamma(qr,s)]} \\
	{[\widehat{q}_*(\widehat{r}\widehat{s})_* V,k^{\widehat{q}\widehat{r}\widehat{s}}\gamma(r,s)^{\widehat{q}}\gamma(q,rs)]} & {[(\widehat{q}\widehat{r})_*\widehat{s}_* V,k^{\widehat{q}\widehat{r}\widehat{s}}\gamma(q,r)\gamma(qr,s)]} \\
	{[\widehat{q}_*\gamma(r,s)_*(\widehat{rs})_* V,k^{\widehat{q}\widehat{r}\widehat{s}}\gamma(r,s)^{\widehat{q}}\gamma(q,rs)]} & {[\gamma(q,r)_*(\widehat{qr})_*\widehat{s}_* V,k^{\widehat{q}\widehat{r}\widehat{s}}\gamma(q,r)\gamma(qr,s)]} \\
	{[\widehat{q}_*\gamma(r,s)_*(\widehat{rs})_* V,\gamma(r,s)^{\widehat{q}}k^{\widehat{q}\widehat{rs}}\gamma(q,rs)]} & {U_{q,r}\cdot[(\widehat{qr})_*\widehat{s}_* V,k^{\widehat{qr}\widehat{s}}\gamma(qr,s)]} \\
	{[(\widehat{q}\gamma(r,s))_*(\widehat{rs})_* V,\gamma(r,s)^{\widehat{q}}k^{\widehat{q}\widehat{rs}}\gamma(q,rs)]} & {U_{q,r}\cdot[(\widehat{qr}\widehat{s})_* V,k^{\widehat{qr}\widehat{s}}\gamma(qr,s)]} \\
	{[(\widehat{q}\gamma(r,s)\widehat{q}^{-1})_*\widehat{q}_*(\widehat{rs})_* V,\gamma(r,s)^{\widehat{q}}k^{\widehat{q}\widehat{rs}}\gamma(q,rs)]} & {U_{q,r}\cdot[\gamma(qr,s)_*(\widehat{qrs})_* V,k^{\widehat{qr}\widehat{s}}\gamma(qr,s)]} \\
	{q_* U_{r,s}\cdot[\widehat{q}_*(\widehat{rs})_* V,k^{\widehat{q}\widehat{rs}}\gamma(q,rs)]} & {U_{q,r}\cdot U_{qr,s}\cdot[(\widehat{qrs})_* V,k^{\widehat{qrs}}]} \\
	{q_* U_{r,s}\cdot[(\widehat{q}\widehat{rs})_* V,k^{\widehat{q}\widehat{rs}}\gamma(q,rs)]} \\
	{q_* U_{r,s}\cdot[\gamma(q,rs)_*(\widehat{qrs})_* V,k^{\widehat{q}\widehat{rs}}\gamma(q,rs)]} & {q_* U_{r,s}\cdot U_{q,rs}\cdot[(\widehat{qrs})_* V,k^{\widehat{qrs}}]}
	\arrow["{\text{id}}", Rightarrow, from=1-1, to=1-2]
	\arrow["{\widehat{q}_*(\chi)}"', Rightarrow, from=1-1, to=2-1]
	\arrow["\chi", Rightarrow, from=1-2, to=2-2]
	\arrow["{\widehat{q}(\chi^{-1})}"', Rightarrow, from=2-1, to=3-1]
	\arrow["{\chi^{-1}}", Rightarrow, from=2-2, to=3-2]
	\arrow["{\text{id}}"', Rightarrow, from=3-1, to=4-1]
	\arrow["{\text{id}}", Rightarrow, from=3-2, to=4-2]
	\arrow["\chi"', Rightarrow, from=4-1, to=5-1]
	\arrow["{1\otimes\chi}", Rightarrow, from=4-2, to=5-2]
	\arrow["{\chi^{-1}}"', Rightarrow, from=5-1, to=6-1]
	\arrow["{1\otimes\chi^{-1}}", Rightarrow, from=5-2, to=6-2]
	\arrow["{\text{id}}"', Rightarrow, from=6-1, to=7-1]
	\arrow["{\text{id}}", Rightarrow, from=6-2, to=7-2]
	\arrow["{1\otimes\chi}"', Rightarrow, from=7-1, to=8-1]
	\arrow["{\text{id}}", Rightarrow, from=7-2, to=9-2]
	\arrow["{1\otimes\chi^{-1}}"', Rightarrow, from=8-1, to=9-1]
	\arrow["{\text{id}}"', Rightarrow, from=9-1, to=9-2]
\end{tikzcd}\]
Omitting the identity maps gives the diagram
\[\begin{tikzcd}[center picture]
	{[\widehat{q}_*\widehat{r}_*\widehat{s}_* V,k^{\widehat{q}\widehat{r}\widehat{s}}\gamma(r,s)^{\widehat{q}}\gamma(q,rs)]} & {[(\widehat{q}\widehat{r})_*\widehat{s}_* V,k^{\widehat{q}\widehat{r}\widehat{s}}\gamma(q,r)\gamma(qr,s)]} \\
	{[\widehat{q}_*(\widehat{r}\widehat{s})_* V,k^{\widehat{q}\widehat{r}\widehat{s}}\gamma(r,s)^{\widehat{q}}\gamma(q,rs)]} && {U_{q,r}\cdot[(\widehat{qr})_*\widehat{s}_* V,k^{\widehat{qr}\widehat{s}}\gamma(qr,s)]} \\
	{[\widehat{q}_*\gamma(r,s)_*(\widehat{rs})_* V,k^{\widehat{q}\widehat{r}\widehat{s}}\gamma(r,s)^{\widehat{q}}\gamma(q,rs)]} & {[(\widehat{q}\widehat{r}\widehat{s})_* V,k^{\widehat{q}\widehat{r}\widehat{s}}\gamma(r,s)^{\widehat{q}}\gamma(q,rs)]} \\
	{[(\widehat{q}\gamma(r,s))_*(\widehat{rs})_* V,\gamma(r,s)^{\widehat{q}}k^{\widehat{q}\widehat{rs}}\gamma(q,rs)]} && {U_{q,r}\cdot[(\widehat{qr}\widehat{s})_* V,k^{\widehat{qr}\widehat{s}}\gamma(qr,s)]} \\
	{q_* U_{r,s}\cdot[\widehat{q}_*(\widehat{rs})_* V,k^{\widehat{q}\widehat{rs}}\gamma(q,rs)]} & {q_* U_{r,s}\cdot[(\widehat{q}\widehat{rs})_* V,k^{\widehat{q}\widehat{rs}}\gamma(q,rs)]} & {q_* U_{r,s}\cdot U_{q,rs}\cdot[(\widehat{qrs})_* V,k^{\widehat{qrs}}]}
	\arrow["\chi", Rightarrow, from=1-1, to=1-2]
	\arrow["{\widehat{q}_*(\chi)}"', Rightarrow, from=1-1, to=2-1]
	\arrow["{\chi^{-1}}", Rightarrow, from=1-2, to=2-3]
	\arrow["\chi", Rightarrow, from=1-2, to=3-2]
	\arrow["{\widehat{q}(\chi^{-1})}"', Rightarrow, from=2-1, to=3-1]
	\arrow["\chi", Rightarrow, from=2-1, to=3-2]
	\arrow["{1\otimes\chi}", Rightarrow, from=2-3, to=4-3]
	\arrow["\chi"', Rightarrow, from=3-1, to=4-1]
	\arrow["{\chi^{-1}}"{pos=0.3}, Rightarrow, from=3-2, to=4-1]
	\arrow["{\chi^{-1}}", Rightarrow, from=3-2, to=4-3]
	\arrow["{\chi^{-1}}", Rightarrow, from=3-2, to=5-2]
	\arrow["{\chi^{-1}}"', Rightarrow, from=4-1, to=5-1]
	\arrow["{1\otimes\chi^{-1}}", Rightarrow, from=4-3, to=5-3]
	\arrow["{1\otimes\chi}", Rightarrow, from=5-1, to=5-2]
	\arrow["{1\otimes\chi^{-1}}", Rightarrow, from=5-2, to=5-3]
\end{tikzcd}\]
where it is clear that each ``square'' commutes. Every morphism in diagram (\ref{Omega axiom}) is known to have coefficients given by associators ($\omega$) and again by Mac Lane \cite{Maclane1971-MACCFT} the two scalars on either side of the diagram must agree. Therefore (\ref{can}) commutes and $\Omega$ has the structure of an invertible modification.
\epp

We now have the necessary tools to prove theorem (\ref{theorem1}).
\bp The map $q\mapsto q_*$ defines a functor from the underlying monoidal category of $3\Gr(Q,\pi)$ to the underlying monoidal category of $\Aut(\BK)$. We have constructed pseudonatural isomorphisms
\[(U_{q,r},\chit_{q,r}):q_*r_*\Rightarrow (qr)_*\]
for all $q,r\in Q$ and invertible modifications
\[\Omega_{q,r,s}:(\alpha_{q,r,s})_*\circ \chit_{q,rs}\circ (\id_{q_*}\otimes \chit_{r,s})\Rrightarrow \chit_{qr,s}\circ (\chit_{q,r}\otimes \id_{s_*})\]
for all $q,r,s\in Q$. Finally, we must check that all trihomomorphism axioms (that apply\footnote{i.e. no source and target considerations nor unitor axioms apply}) are satisfied. The lone trihomomorphism axiom that must be checked is for all 1-cells $q,r,s,t\in 3\Gr(Q,\pi)$ the following equation of modifications must hold (where some subscripts and tensor products are omitted and we refer to the identity transform and modification as 1 and we color the modifications for clarity):

\[\begin{tikzcd}[scale cd=1.25, center picture]
	&& {\Big(q\big(r(st)\big)\Big)_*} \\
	& {q_*\big(r(st)\big)_*} & {q_*((rs)t)_*} & {\Big(q\big((rs)t\big)\Big)_*} \\
	{q_*\big(r_*(st)_*\big)} && {q_*\big((rs)_*t_*\big)} && {\Big(\big(q(rs)\big)t\Big)_*} \\
	{q_*(r_*(s_*t_*))} & {q_*\big((r_*s_*)t_*\big)} & {\big(q_*(rs)_*\big)t_*} & {\big(q(rs)\big)_*t_*} & {\Big(\big((qr)s\big)t\Big)_*} \\
	{(q_*r_*)(s_*t_*)} & {\big(q_*(r_*s_*)\big)t_*} &&& {\big((qr)s\big)_*t_*} \\
	& {\big((q_*r_*)s_*\big)t_*} && {\big((qr)_*s_*\big)t_*} \\
	\\
	&& {\Big(q\big(r(st)\big)\Big)_*} \\
	& {q_*\big(r(st)\big)_*} && {\Big(q\big((rs)t\big)\Big)_*} \\
	{q_*\big(r_*(st)_*\big)} && {\big((qr)(st)\big)_*} && {\Big(\big(q(rs)\big)t\Big)_*} \\
	{q_*(r_*(s_*t_*))} & {(q_*r_*)(st)_*} & {(qr)_*(st)_*} && {\Big(\big((qr)s\big)t\Big)_*} \\
	{(q_*r_*)(s_*t_*)} && {(qr)_*(s_*t_*)} && {\big((qr)s\big)_*t_*} \\
	& {\big((q_*r_*)s_*\big)t_*} && {\big((qr)_*s_*\big)t_*}
	\arrow[""{name=0, anchor=center, inner sep=0}, "{(1\alpha_{r,s,t})_*}", Rightarrow, from=1-3, to=2-4]
	\arrow[""{name=1, anchor=center, inner sep=0}, "\chi", Rightarrow, from=2-2, to=1-3]
	\arrow["{1(\alpha_{r,s,t})_*}"', Rightarrow, from=2-2, to=2-3]
	\arrow["{1\Omega_{r,s,t}}"', color={rgb,255:red,214;green,92;blue,92}, shorten <=13pt, shorten >=13pt, Rightarrow, scaling nfold=3, from=2-2, to=4-2]
	\arrow["\chi", Rightarrow, from=2-3, to=2-4]
	\arrow["{(\alpha_{q,rs,t})_*}", Rightarrow, from=2-4, to=3-5]
	\arrow["{\Omega_{q,rs,t}}"', color={rgb,255:red,214;green,92;blue,92}, shorten <=12pt, shorten >=12pt, Rightarrow, scaling nfold=3, from=2-4, to=4-4]
	\arrow["{1\chi}"{description}, Rightarrow, from=3-1, to=2-2]
	\arrow["{1\chi}", Rightarrow, from=3-3, to=2-3]
	\arrow[Rightarrow, no head, from=3-3, to=4-3]
	\arrow["{(\alpha_{q,r,s}1)_*}", Rightarrow, from=3-5, to=4-5]
	\arrow["{1(1\chi)}"', Rightarrow, from=4-1, to=3-1]
	\arrow[Rightarrow, no head, from=4-1, to=4-2]
	\arrow[""{name=2, anchor=center, inner sep=0}, Rightarrow, no head, from=4-1, to=5-1]
	\arrow["{1(\chi 1)}", Rightarrow, from=4-2, to=3-3]
	\arrow["{=}"{marking, allow upside down}, draw=none, from=4-2, to=4-3]
	\arrow[""{name=3, anchor=center, inner sep=0}, Rightarrow, no head, from=4-2, to=5-2]
	\arrow["{\chi 1}", Rightarrow, from=4-3, to=4-4]
	\arrow["\chi", Rightarrow, from=4-4, to=3-5]
	\arrow["{=}"{marking, allow upside down}, draw=none, from=4-4, to=4-5]
	\arrow["{(\alpha_{q,r,s})_* 1}"', Rightarrow, from=4-4, to=5-5]
	\arrow[Rightarrow, no head, from=5-1, to=6-2]
	\arrow["{(1\chi)1}"', Rightarrow, from=5-2, to=4-3]
	\arrow[Rightarrow, no head, from=5-2, to=6-2]
	\arrow["\chi"', Rightarrow, from=5-5, to=4-5]
	\arrow[""{name=4, anchor=center, inner sep=0}, "{(\chi 1)1}"', Rightarrow, from=6-2, to=6-4]
	\arrow["{\chi 1}"', Rightarrow, from=6-4, to=5-5]
	\arrow["{(1\alpha_{r,s,t})_*}", Rightarrow, from=8-3, to=9-4]
	\arrow["{(\alpha_{q,r,st})_*}", Rightarrow, from=8-3, to=10-3]
	\arrow["\chi", Rightarrow, from=9-2, to=8-3]
	\arrow["{\Omega_{q,r,st}}", color={rgb,255:red,214;green,92;blue,92}, shorten <=13pt, shorten >=13pt, Rightarrow, scaling nfold=3, from=9-2, to=11-2]
	\arrow["{(\alpha_{q,rs,t})_*}", Rightarrow, from=9-4, to=10-5]
	\arrow["{1\chi}", Rightarrow, from=10-1, to=9-2]
	\arrow[""{name=5, anchor=center, inner sep=0}, "{(\alpha_{qr,s,t})_*}", Rightarrow, from=10-3, to=11-5]
	\arrow["{(\alpha_{q,r,s}1)_*}", Rightarrow, from=10-5, to=11-5]
	\arrow["{1(1\chi)}", Rightarrow, from=11-1, to=10-1]
	\arrow["{=}"{description}, draw=none, from=11-1, to=11-2]
	\arrow[Rightarrow, no head, from=11-2, to=10-1]
	\arrow["\chi1", Rightarrow, from=11-2, to=11-3]
	\arrow["\chi", Rightarrow, from=11-3, to=10-3]
	\arrow[Rightarrow, no head, from=12-1, to=11-1]
	\arrow[""{name=6, anchor=center, inner sep=0}, "{1\chi}"', Rightarrow, from=12-1, to=11-2]
	\arrow["\chi1"', Rightarrow, from=12-1, to=12-3]
	\arrow[""{name=7, anchor=center, inner sep=0}, Rightarrow, no head, from=12-1, to=13-2]
	\arrow[""{name=8, anchor=center, inner sep=0}, "{1\chi}", Rightarrow, from=12-3, to=11-3]
	\arrow[""{name=9, anchor=center, inner sep=0}, Rightarrow, no head, from=12-3, to=13-4]
	\arrow["\chi"', Rightarrow, from=12-5, to=11-5]
	\arrow["{(\chi1)1}"', Rightarrow, from=13-2, to=13-4]
	\arrow["\chi1"', Rightarrow, from=13-4, to=12-5]
	\arrow["{=}"{description}, draw=none, from=1, to=0]
	\arrow["{=}"{description}, draw=none, from=2, to=3]
	\arrow["{\Omega_{q,r,s} 1}", color={rgb,255:red,214;green,92;blue,92}, shorten <=15pt, shorten >=15pt, Rightarrow, scaling nfold=3, from=4-3, to=4]
	\arrow["{=}"{marking, allow upside down, pos=0.6}, draw=none, from=4, to=8-3]
	\arrow["{(\pi_{q,r,s,t})_*}"'{pos=0.4}, color={rgb,255:red,92;green,92;blue,214}, shorten <=8pt, shorten >=13pt, Rightarrow, scaling nfold=3, from=9-4, to=5]
	\arrow["{\Omega_{qr,s,t}}"', color={rgb,255:red,214;green,92;blue,92}, shorten <=19pt, shorten >=19pt, Rightarrow, scaling nfold=3, from=5, to=13-4]
	\arrow["{=}"{description}, draw=none, from=6, to=8]
	\arrow["{=}"{description}, draw=none, from=7, to=9]
\end{tikzcd}\]
\newpage

By the functorality of $*$ we have that $(\pi_{q,r,s,t})_*$ is $\pi(q,r,s,t)\id$, where $\id$ is the identity modification, in $\Aut(\BK)$. By definition of $\Omega$, the above equality of modifications holds if and only if 
\[\omega(\ar,\s,\tee)\omega(\q,\ar\s,\tee)\omega(\q,\ar,\s)=\pi(q,r,s,t)\omega(\q,\ar,\s\tee)\omega(\q\ar,\s,\tee)\]
but $\rho(\q)=q$, $\rho(\ar)=r$, $\rho(\s)=s$ and $\rho(\tee)=t$ by definition of the lift function. Therefore the above equality holds by equation (\ref{eq2.1}).
\epp

We now investigate concrete examples of anomalous actions of groups on tensor categories using our main theorem.

\section{Examples}
\subsection{Free groups}
Our first class of examples, the easiest of which to apply our main theorem, involves free groups. If we denote $\F_n$ the free group with $n$ generators, then we have $H^k(\F_n,\kt)=1$ for all $k\geq 2$. For any finitely generated $Q$ and $\pi\in Z^4(Q,\kt)$, we can find a surjection $\rho:\F_n\rightarrow Q$ which is completely determined by picking $n$ elements in $Q$ which generate. The kernel of $\rho$ is a free normal subgroup $K\leq \F_n$, with $K\cong \F_m$ with $m$ possibly $\infty$. By \cite{lyndon2001combinatorial} (proposition 3.9), if $|Q|<\infty$ and $n>1$, then we have
\[m=|Q|(n-1)+1.\]
Since $H^4(\F_n,\kt)=1=H^3(\F_m,\kt)$, we can find a trivialization $\omega\in C^3(\F_n,\kt)$ of $\rho^*(\pi)$ with $\omega|_{K\cong\F_m}=1$. Applying the main theorem gives the following:
\begin{corollary}
    Let $Q$ be a finite group. Pick a surjection $\rho:\F_n\rightarrow Q$ with $n>1$. Now, given an $\F_n$-action on a tensor category $\mathcal{C}$ and $\pi\in Z^4(Q,\kt)$ then there exists a $\pi$-anomalous $Q$ action on $\mathcal{C}\rtimes \F_m$ with $m=|Q|(n-1)+1$.
\end{corollary}

\begin{exmp}
    Let $\F_n$ be a free group with $n>1$ and let $\{f_i\}$ be a set of homeomorphisms of the topological space $X$. Then we freely define an action of $\F_n$ on $X$ by sending generators of $\F_n$ to the set $\{f_i\}$. Now $\F_n$ acts on the category of vector bundles over $X$, $\text{Vect}(X)$ (using the terminology of \cite{Hatcher} we take $\text{Vect}(X)$ to be the restriction of the category of fiber bundles over a space $X$, referred to as $\text{bun}/X$ in \cite{alma99442987900561}, to vector bundles over a space $X$). By the previous corollary, for a finite group $Q$ with surjection $\rho:\F_n\rightarrow Q$ and $\pi\in Z^4(Q,\kt)$ we have a $\pi$-anomalous $Q$ action on $\text{Vect}(X)\rtimes \F_m$.
\end{exmp}

\subsection{Finite groups}
We have found that the easiest method for creating examples involves cases when we are able to decompose non-trivial 4-cocycles into a cup product of two 2-cocycles or a 2-cocycle and two 1-cocycles.

Consider the short exact sequence
\[0\rightarrow(\F_p,+)\hookrightarrow (\C^{\times},\times)\twoheadrightarrow (\C^{\times},\times)\rightarrow 0\]
where we identify $\F_p\cong \Z_p$ with the $p$-roots of unity in $\C^{\times}$. Then we have the long exact sequence
\[\cdots\rightarrow H^n(Q,\C^{\times})\xrightarrow{p}H^n(Q,\C^{\times})\xrightarrow{\delta}H^{n+1}(Q,\F_p)\rightarrow H^{n+1}(Q,\C^{\times})\rightarrow \cdots\]
We want to produce non-trivial 4-cocycles $\pi\in Z^4(Q,\C^{\times})$ while using 4-cocycles from $Z^4(Q,\F_p)$ in our construction, hence we want $H^4(Q,\F_p)\rightarrow H^4(Q,\C^{\times})$ to be injective.
We then need that $p$ is coprime to $|H^3(Q,\C^{\times})|$. The ring structure of $\F_p$ allows the use of cup products. If we can find $\pi\in Z^4(Q,\F_p)$ that decomposes into a cup product $\pi=\sigma c$ where $\sigma\in Z^2(Q,\F_p)$ then we can always construct the surjection 
\[G:=\F_p\rtimes_{\sigma} Q\xrightarrow{\rho} Q.\]
We can then directly trivialize any 2-cocycle $\rho^*(\sigma)=0$ and thus 
\[\rho^*(\pi)=\rho^*(\sigma c)=\rho^*(\sigma)\rho^*(c)=0\]
by the properties of the cup product on cohomology.
\begin{exmp}
    By \cite{Gannon_2019}, there is exactly one non-trivial class in $H^5(S_4,\Z)\cong H^4(S_4,\C^{\times})$. 

    Consider the short exact sequence
    \[0\rightarrow \Z \xrightarrow{\times 2}\Z \xrightarrow{\pi}\F_2\rightarrow 0\]
    that gives the long exact sequence in cohomology
    \[\cdots \rightarrow H^k(Q,\Z)\xrightarrow{\times 2}H^k(Q,\Z)\xrightarrow{\pi}H^k(Q,\Z_2)\xrightarrow{\beta}H^{k+1}(Q,\Z)\rightarrow\cdots\]

    The cohomology ring $H^*(S_4,\F_2)$ is $\F_2[\sigma_1,\sigma_2,c_3]/(\sigma_1 c_3)$ where subscripts denote degree and so $H^4(S_4,\F_2)$ is spanned by $\sigma_1^4,\sigma_1^2\sigma_2$ and $\sigma_2^2$.
    The unique non-trivial class in $H^5(S_4,\Z)$ can be identified (through the isomorphism $H^5(S_4,\Z)\cong Sq^1(Z^4(S_4,\F_2))$) with the element $Sq^1(\sigma_1^2\sigma_2)=\sigma_1^3\sigma_2$, where $Sq^1:H^k(Q,\F_2)\rightarrow H^{k+1}(Q,\F_2)$ denotes the \textit{Steenrod square} (see page 11 of \cite{Gannon_2019} for more details.) We will also identify, through the isomorphism $H^5(S_4,\Z)\cong H^4(S_4,\C^\times)$, the 4-cocycle $\tilde{\pi}\in H^4(S_4,\C^\times)$ with the element $Sq^1(\sigma_1^2\sigma_2)$. The element $\sigma_2\in Z^2(S_4,\F_2)$ is our twist
    \[\Z_2\rtimes_{\sigma_2} S_4\xrightarrow{\rho} S_4.\]
    We can directly trivialize $\rho^*(\sigma_2)\in H^2(\Z_2\rtimes_{\sigma_2} S_4,\F_2)$, i.e., define $c:\Z_2\rtimes_{\sigma_2}S_4\rightarrow \F_2$ by
    \[c(a,b):=a\]
    for $a\in \Z_2$ and $b\in S_4$. A direct calculation shows that $\delta c=\rho^* (\sigma_2)$ and so 
    \[\rho^*(\sigma_1^3 \sigma_2)=0.\]

    Now we would like to build an action of $\Z_2\rtimes_{\sigma_2}S_4$ on a tensor category $\mathcal{C}$. For any symmetric group $S_n$, we have the standard categorical permutation action on $\mathcal{C}^{\boxtimes n}$. We also can, in general, twist an action $G\rightarrow \aut (\mathcal{C})$ by a 2-cocycle $c\in Z^2(G,\aut(1_{\mathcal{C}}))$ where $\aut(1_{\mathcal{C}})$ is the group of tensor automorphisms of the identity functor in $\mathcal{C}$. By \cite{etingof2016tensor} (proposition 4.14.3), we have a canonical isomorphism between $\aut(1_{\mathcal{C}})$ and characters of the grading group of $\mathcal{C}$. Therefore if the grading group has $\Z_2$ as a subgroup, for instance $\Z_2$ itself, then $\sigma_2$ is such a 2-cocycle. The twist is given by the composition
    \[g_*\circ h_*\cong (gh)_*\xrightarrow{\sigma_2(g,h)1_{(gh)_*}}(gh)_*\]
    Therefore if we have an $H$-graded fusion category $\mathcal{C}$, where $\Z_2\leq H$,  (for instance a Tambara-Yamagami category \cite{Tambara1998TensorCW}), then we have a $\sigma_2$-twisted permutation action of $\Z_2\rtimes_{\sigma_2} S_4$ on $\mathcal{C}^{\boxtimes 4}$. Thus if we pick a trivialization $d\omega=\rho^*(\sigma_1^3\sigma_2)$ we can build a $\tilde{\pi}$-anomalous action of $S_4$ on $\mathcal{C}^{\boxtimes 4}\rtimes_{\omega} \Z_2$. 
\end{exmp}

\newpage

Further generalizing this example we may say:
\begin{itemize}
    \item  given any group $G=\Z_n\rtimes_c Q$ with $\pi\in Z^4(Q,\Z_n)$ and a decomposition of $\pi=cc'$ where $c\in Z^2(Q,\Z_n)$,
    \item a tensor category $\mathcal{C}$ that is $\Gamma$-graded such that $\Z_n\leq \Gamma$
    \item an action of $G$ on the tensor category $\mathcal{C}$, $\underline{G}\rightarrow\aut(\mathcal{C})$,
\end{itemize}
then we may always twist the $G$-action by $c$ and therefore have a new action of $G$ on $\mathcal{C}$, which then (after picking a trivialization $d\omega=\rho^*(\pi)$) gives a $\pi$-anomalous action of $Q$ on $\mathcal{C}\rtimes_{\omega}\Z_n$.

\bibliographystyle{abbrv}
\bibliography{bib.bib}
\end{document}